\documentclass{amsart}
\usepackage[normalem]{ulem}
\usepackage[all]{xy}
\usepackage{verbatim}
\usepackage{color}
\usepackage{amsthm}
\usepackage{amssymb}
\usepackage[colorlinks=true]{hyperref}

\numberwithin{equation}{section}

\newtheorem{theorem}[equation]{Theorem}
\newtheorem*{theorem*}{Theorem} \newtheorem{lemma}[equation]{Lemma}

\newtheorem*{conjecture*}{Mamma Conjecture}
\newtheorem*{conjecture1*}{Mamma Conjecture (revisited)}
\newtheorem{proposition}[equation]{Proposition}
\newtheorem{corollary}[equation]{Corollary}
\newtheorem*{corollary*}{Corollary}

\theoremstyle{remark}
\newtheorem{definition}[equation]{Definition}

\newtheorem{example}[equation]{Example}

\newtheorem{notation}[equation]{Notation}

\theoremstyle{remark}
\newtheorem{remark}[equation]{Remark}

\newcommand{\cA}{{\mathcal A}}
\newcommand{\cB}{{\mathcal B}}
\newcommand{\cC}{{\mathcal C}}
\newcommand{\cD}{{\mathcal D}}
\newcommand{\cE}{{\mathcal E}}
\newcommand{\cF}{{\mathcal F}}
\newcommand{\cG}{{\mathcal G}}
\newcommand{\cH}{{\mathcal H}}

\newcommand{\cL}{{\mathcal L}}

\newcommand{\cO}{{\mathcal O}}

\newcommand{\cT}{{\mathcal T}}

\newcommand{\cX}{{\mathcal X}}
\newcommand{\cY}{{\mathcal Y}}
\newcommand{\cZ}{{\mathcal Z}}

\newcommand{\Spt}{\mathrm{Spt}}

\newcommand{\bbA}{\mathbb{A}}

\newcommand{\bbC}{\mathbb{C}}

\newcommand{\bbF}{\mathbb{F}}

\newcommand{\bbQ}{\mathbb{Q}}
\newcommand{\bbZ}{\mathbb{Z}}

\DeclareMathOperator{\id}{id}

\DeclareMathOperator{\NMot}{NMot}
\DeclareMathOperator{\NChow}{NChow} 

\DeclareMathOperator{\Fun}{Fun} 

\newcommand{\dgcat}{\mathrm{dgcat}} 

\newcommand{\bbK}{I\mspace{-6.mu}K}

\newcommand{\perf}{\mathrm{perf}}

\newcommand{\dg}{\mathrm{dg}}

\newcommand{\Hom}{\mathrm{Hom}}
\newcommand{\End}{\mathrm{End}}

\newcommand{\rep}{\mathrm{rep}}

\newcommand{\dgHo}{\mathrm{H}^0}

\newcommand{\Ho}{\mathrm{Ho}}

\newcommand{\Hmo}{\mathrm{Hmo}}
\newcommand{\HHmo}{\mathbb{H}\mathrm{mo}}
\newcommand{\op}{\mathrm{op}}

\newcommand{\Map}{\mathrm{Map}}

\newcommand{\too}{\longrightarrow}

\newcommand{\ie}{\textsl{i.e.}\ }
\newcommand{\eg}{\textsl{e.g.}}

\newcommand{\U}{\mathrm{U}}

\newcommand{\MMot}{\mathbb{M}\mathrm{ot}}
\newcommand{\UU}{\mathbb{U}}
\newcommand{\vphi}{{\varphi\!}{/}{\!\sim}}

\def\r{\rightarrow}

\let\oldmarginpar\marginpar
\def\marginpar#1{\oldmarginpar{\tiny #1}}

\def\pr{\operatorname{pr}}

\def\Spec{\operatorname{Spec}}
\def\DM{\operatorname{DM}}
\def\coh{\operatorname{coh}}

\let\Oscr\cO

\def\rad{\operatorname{rad}}

\begin{document}

\title[Additive invariants of orbifolds]{Additive invariants of orbifolds}
\author{Gon{\c c}alo~Tabuada and Michel Van den Bergh}

\address{Gon{\c c}alo Tabuada, Department of Mathematics, MIT, Cambridge, MA 02139, USA}
\email{tabuada@math.mit.edu}
\urladdr{http://math.mit.edu/~tabuada}
\thanks{G.~Tabuada was partially supported by a NSF CAREER Award.}
\thanks{M.~Van den Bergh is a senior researcher at the Research Foundation -- Flanders}

\address{Michel Van den Bergh, Department of Mathematics, Universiteit Hasselt, 3590 Diepenbeek, Belgium}
\email{michel.vandenbergh@uhasselt.be}
\urladdr{http://hardy.uhasselt.be/personal/vdbergh/Members/~michelid.html}

\date{\today}


\abstract{In this article, using the recent theory of
  noncommutative motives, we compute the additive invariants of orbifolds (equipped with a sheaf of Azumaya algebras) using solely ``fixed-point data''. As a consequence, we recover, in a unified
    and conceptual way, the original results of Vistoli concerning algebraic $K$-theory, of Baranovsky concerning cyclic homology, of the second author and Polishchuk concerning Hochschild homology, and of Baranovsky-Petrov and C\v{a}ld\v{a}raru-Arinkin (unpublished) concerning twisted Hochschild homology; in the case of topological Hochschild homology and periodic topological cyclic homology, the aforementioned computation is new in the literature. As an application, we verify Grothendieck's
    standard conjectures of type $C^+$ and $D$, as well as Voevodsky's
    smash-nilpotence conjecture, in the case of ``low-dimensional'' orbifolds. Finally, we establish a result of independent interest concerning nilpotency in the Grothendieck ring of~an~orbifold. 
}}

\maketitle
\vskip-\baselineskip
\vskip-\baselineskip
\tableofcontents

\vskip-\baselineskip
\vskip-\baselineskip
\vskip-\baselineskip


\section{Introduction}\label{sec:intro}
A {\em differential graded (=dg) category $\cA$}, over a base field
$k$, is a category enriched over complexes of $k$-vector spaces; see
\S\ref{sub:dg}. Let us denote by $\dgcat(k)$ the category of
(essentially small) $k$-linear dg categories. Every (dg) $k$-algebra
$A$ gives naturally rise to a dg category with a single
object. Another source of examples is provided by schemes since the
category of perfect complexes $\perf(X)$ of every
quasi-compact 
quasi-separated $k$-scheme $X$ (or, more generally, suitable algebraic
stack $\cX$) admits a canonical dg enhancement $\perf_\dg(X)$; consult
\cite[\S4.6]{ICM-Keller}\cite{LO}. Moreover, the tensor product
$-\otimes_X-$ makes $\perf_\dg(X)$ into a commutative monoid in the category obtained from
$\dgcat(k)$ by inverting the Morita equivalences; see \cite{Schnurer}.

An {\em additive invariant} of dg categories is a functor $E\colon \dgcat(k) \to \mathrm{D}$, with values in an idempotent complete additive category, which inverts Morita equivalences and sends semi-orthogonal decompositions to direct sums; consult \S\ref{sec:additive} for further details. As explained in {\em loc. cit.}, examples include several variants of algebraic $K$-theory, of cyclic homology, and of  topological Hochschild homology. Given a $k$-scheme $X$ (or stack) as above, let us write $E(X)$ instead of $E(\perf_\dg(X))$. Note that if $E$ is lax (symmetric) monoidal, then $E(X)$ is a (commutative) monoid in $\mathrm{D}$. 

Now, let $G$ be a finite group acting on a smooth separated $k$-scheme $X$. In what follows, we will write $[X/G]$, resp. $X/\!\!/G$, for the associated (global) orbifold, resp. geometric quotient. The results of this article may be divided into three parts:
\begin{itemize}
\item[(i)] {\em Decomposition of orbifolds:} We establish some formulas for $E([X/G])$ in terms of fixed-point data $\{E(X^g)\}_{g \in G}$; consult Theorem \ref{thm:main} and Corollaries \ref{cor:main1} and \ref{cor:main2}. In the particular case where $E$ is algebraic $K$-theory, resp. cyclic homology, these formulas reduce to previous results of Vistoli, resp. Baranovsky.
\item[(ii)] {\em Smooth quotients:} We prove that if $X/\!\!/G$ is smooth, then $E(X/\!\!/G)\simeq E(X)^G$; consult Theorem \ref{thm:main2}. In the particular case where $E$ is Hochschild homology, this reduces to a previous result of the second author with Polishchuk.
\item[(iii)] {\em Equivariant Azumaya algebras:} We extend the formulas of the above part (i) to the case where $X$ is equipped with a $G$-equivariant sheaf of Azumaya algebras; consult Theorem \ref{thm:formula-main4} and Corollaries \ref{cor:az1}, \ref{cor:az3}, and \ref{cor:az4}. In the particular case where $E$ is Hochschild homology, these formulas reduce to an earlier result of Baranovsky-Petrov and C\v{a}ld\v{a}raru-Arinkin (unpublished).   
\end{itemize}
As an application of the formulas established in part (i), we verify Grothendieck's
    standard conjectures of type $C^+$ and $D$, as well as Voevodsky's
    smash-nilpotence conjecture, in the case of ``low-dimensional'' orbifolds; consult Theorem \ref{thm:conjectures}. 
\subsection*{Statement of results}
Let $k$ be a base field of characteristic $p\geq0$, $G$ a finite group
of order $n$, $\varphi$ the set of all cyclic subgroups of $G$, and 
$\vphi$ a (chosen) set of representatives of the conjugacy classes in $\varphi$. Given a cyclic subgroup $\sigma \in \varphi$, we will write $N(\sigma)$ for the normalizer of $\sigma$. Throughout the article, we will assume that $1/n\in k$; we will {\em not} assume that $k$ contains the $n^{\mathrm{th}}$ roots of unity. 

Let $X$ be a smooth separated $k$-scheme equipped with a $G$-action; we will {\em not} assume that $X$ is quasi-projective. As above, we will write $[X/G]$, resp. $X/\!\!/G$, for the associated (global) orbifold, resp. geometric quotient.
%
%
%
%
%
    
\subsubsection*{Decomposition of orbifolds}
Let $R(G)$ be the representation ring of $G$. As explained in \S\ref{sec:action}, the assignment $[V] \mapsto V\otimes_k -$, where $V$ stands for a $G$-representation, gives rise to an action of $R(G)$ on $E([X/G])$ for every additive~invariant~$E$. 

Given a cyclic subgroup $\sigma \in \varphi$, recall from Definition \ref{def:primitive} below that the $\bbZ[1/n]$-linearized representation ring $R(\sigma)_{1/n}$ comes equipped with a certain canonical idempotent $e_\sigma$. As explained in {\em loc. cit.}, $e_\sigma$ can be characterized as the maximal idempotent whose image under all the restrictions $R(\sigma)_{1/n} \to R(\sigma')_{1/n}, \sigma' \subsetneq \sigma$, is zero. Whenever an object $\cO$ (of an idempotent complete category) is equipped with an $R(\sigma)_{1/n}$-action, we will write $\widetilde{\cO}$ for the direct summand $e_\sigma \cO$. In particular, $\widetilde{R}(\sigma)_{1/n}$ stands for the direct summand $e_\sigma R(\sigma)_{1/n}$.

Let $E\colon \dgcat(k) \to \mathrm{D}$ be an additive invariant with values in a $\bbZ[1/n]$-linear category. On the one hand, as mentioned above, we have a canonical $R(\sigma)_{1/n}$-action on $E([X^\sigma/\sigma])$. On the other hand, $N(\sigma)$ acts naturally on $[X^\sigma/\sigma]$ and hence on $E([X^\sigma/\sigma])$. By functoriality, the former action is compatible with the latter. Consequently, we obtain an induced $N(\sigma)$-action on $\widetilde{E}([X^\sigma/\sigma])$. Under the above notations, our first main result is the following:
\begin{theorem}\label{thm:main} 
For every additive invariant $E\colon \dgcat(k) \to \mathrm{D}$, with values in a 
$\bbZ[1/n]$-linear category, we have an isomorphism
\begin{equation}\label{eq:formula-main}
\begin{aligned}
E([X/G]) &\simeq \bigoplus_{\sigma \in \vphi} \widetilde{E}([X^\sigma/\sigma])^{N(\sigma)}
\end{aligned}
\end{equation}
induced by pull-back with respect to the morphisms $[X^\sigma/\sigma]\r [X/G]$.
Moreover, if $E$ is lax monoidal, then \eqref{eq:formula-main} is an isomorphism of monoids.  
\end{theorem}
\begin{remark} The right-hand side of \eqref{eq:formula-main} may be re-written as $(\bigoplus_{\sigma \in \varphi} \widetilde{E}([X^\sigma/\sigma]))^G$. In this way, the isomorphism \eqref{eq:formula-main} does not depend on any choices.
\end{remark}
\begin{remark} Note that $\sigma$ acts trivially on $X^\sigma$. Therefore, by combining Proposition \ref{prop:generator} below with \cite[Lem.~4.26]{Gysin}, we conclude that the dg category $\perf_{\dg} ([X^\sigma/\sigma])$ is Morita equivalent to $\perf_{\dg}(X^\sigma\times\Spec(k[\sigma]))$.
%
This implies that the above isomorphism \eqref{eq:formula-main} can be re-written as follows:
\begin{equation}
\label{eq:formula-main-sub}
\begin{aligned}
E([X/G]) 
&\simeq \bigoplus_{\sigma \in \vphi} \widetilde{E}(X^{\sigma}\times \Spec(k[\sigma]))^{N(\sigma)} 
\,.
\end{aligned}
\end{equation}
Intuitively speaking, \eqref{eq:formula-main-sub} shows that every additive 
invariant of (global) orbifolds can be computed using solely ``ordinary''
 schemes. Unfortunately, when $E$ is lax monoidal the above isomorphism \eqref{eq:formula-main-sub} obscures the monoid structure on $E([X/G])$.
 \end{remark} 
 
 Given a commutative ring $R$ and an $R$-linear idempotent complete additive category $\mathrm{D}$, let us write $-\otimes_R-$ for the canonical action of the category of finitely generated projective $R$-modules on $\mathrm{D}$. Under some mild assumptions, the above Theorem \ref{thm:main} admits the following refinements:
%
%

\begin{corollary}\label{cor:main1}
\begin{itemize}
\item[(i)] If $k$ contains the $n^{\mathrm{th}}$ roots of
  unity, then
  \eqref{eq:formula-main} reduces to
\begin{equation}\label{eq:formula-main2}
 E([X/G]) \simeq \bigoplus_{\sigma \in \vphi} (E(X^\sigma) \otimes_{\bbZ[1/n]}
\widetilde{R}(\sigma)_{1/n})^{N(\sigma)}\,.
\end{equation}
\item[(ii)] If $k$ contains the $n^{\mathrm{th}}$ roots of
  unity and $\mathrm{D}$ is $l$-linear for a field $l$ which contains the $n^{\mathrm{th}}$ roots of unity and $1/n\in l$, then \eqref{eq:formula-main2} reduces to an isomorphism
\begin{equation}\label{eq:formula-main3}
 E([X/G]) \simeq \bigoplus_{g\in G\!/\!\sim} E(X^g)^{C(g)} \simeq (\bigoplus_{g\in G} E(X^g))^G\,,
\end{equation}
where $C(g)$ stands for the centralizer of $g$.
\end{itemize}
Moreover, if $E$ is lax monoidal, then \eqref{eq:formula-main2}-\eqref{eq:formula-main3} are isomorphisms of monoids.
\end{corollary}
\begin{corollary}\label{cor:main2}
If $E$ is monoidal, then \eqref{eq:formula-main} reduces to an isomorphism of monoids
\begin{equation*}
E([X/G]) \simeq \bigoplus_{\sigma \in \vphi} (E(X^\sigma) \otimes \widetilde{E}(B\sigma))^{N(\sigma)}\,,
\end{equation*}
where $B\sigma:=[\bullet/\sigma]$ stands for the classifying stack of $\sigma$.
\end{corollary}
\begin{example}[Algebraic $K$-theory]\label{ex:algebraic1}
As mentioned in Example \ref{ex:K-theory} below, algebraic $K$-theory is a lax symmetric monoidal additive invariant. Therefore, in the case where $k$ contains the $n^{\mathrm{th}}$ roots of unity, Corollary \ref{cor:main1}(i) leads to the following isomorphism of $\bbZ$-graded commutative $\bbZ[1/n]$-algebras\footnote{As Nick Kuhn kindly informed us, in the topological setting the above formula \eqref{eq:Vistoli} goes back to the pioneering work of tom Dieck \cite[\S7.7]{Dieck}; see also \cite[\S6]{Kuhn}.}:
\begin{equation}\label{eq:Vistoli}
K_\ast([X/G])_{1/n} \simeq \bigoplus_{\sigma \in \vphi} (K_\ast(X^\sigma)_{1/n} \otimes_{\bbZ[1/n]}
\widetilde{R}(\sigma)_{1/n})^{N(\sigma)}\,.
\end{equation}
Vistoli established the formula \eqref{eq:Vistoli} in \cite[Thm.~1]{Vistoli} under the weaker assumptions that $X$ is regular, Noetherian, and of finite Krull dimension, but under the additional assumption 
that $X$ carries an ample line bundle\footnote{The assumption that
$X$ carries an ample line bundle  can 
be removed without affecting the validity of Vistoli's results; consult Remark \ref{rk:ample} below.}. Moreover, since he did not assume that $1/n \in k$, he excluded from the direct sum the (conjugacy classes of) cyclic subgroups whose order is divisible by $p$. Note that Corollary \ref{cor:main1}(i) enables us to upgrade \eqref{eq:Vistoli} to an homotopy equivalence of spectra.
\end{example}
\begin{example}[Mixed complex]\label{ex:mixed2}
As mentioned in Example \ref{ex:mixed} below, the mixed complex ${{\mathsf{C}}}$ is a symmetric monoidal additive invariant. Therefore, in the case where $k$ contains the $n^{\mathrm{th}}$ roots of unity, Corollary \ref{cor:main1}(ii) leads to the following isomorphism of commutative monoids in the derived category of mixed complexes:
\begin{equation}\label{eq:Baranovsky}
{{\mathsf{C}}}([X/G]) \simeq (\bigoplus_{g \in G} {{\mathsf{C}}}(X^g))^G\,.
\end{equation}
Note that since ${{\mathsf{C}}}$ is compatible with base-change and
the morphism \eqref{eq:Baranovsky}, induced by pull-back, is defined
over $k$, the assumption that $k$ contains the $n^{\mathrm{th}}$ roots
of unity can be removed!
Baranovsky established the above formula \eqref{eq:Baranovsky} in
\cite[Thm.~1.1 and Prop.~3.1]{Baranovsky} under the additional
assumption that $X$ is quasi-projective. Moreover, he used
$G$-coinvariants instead of $G$-invariants. Since $1/n\in k$,
$G$-coinvariants and $G$-invariants are canonically isomorphic. Hence,
the difference between \eqref{eq:Baranovsky} and Baranovsky's formula
is only ``cosmetic''; see also \cite[Cor.\ 1.17(2)]{Caldararu5}. The
advantage of $G$-invariants over $G$-coinvariants is that 
the former construction preserves monoid structures.
\end{example}
\begin{remark}[Orbifold cohomology]
Cyclic homology and its
variants factor through ${{\mathsf{C}}}$. Consequently, a formula similar to \eqref{eq:Baranovsky} holds for all these
invariants. 
For example, when $k=\bbC$ and $E$ is periodic cyclic homology $H\!P_\ast$, the
Hochschild-Kostant-Rosenberg theorem leads to an isomorphism\footnote{Let $\cX:=[X/G]$ be the (global) orbifold and $I(\cX):=[(\amalg_{g \in G} X^g)/G]$ the inertia stack of $\cX$. Under these notations, the above isomorphism \eqref{eq:inertia} can be re-written as $H\!P_\ast(\cX)\simeq H^\ast(I(\cX),\bbC)$. Halpern-Leistner and Pomerleano \cite{HLP}, and To\"en (unpublished), using the techniques in \cite{Toen2}, extended the latter isomorphism to all smooth Deligne-Mumford stacks $\cX$.} of $\bbZ/2$-graded $\bbC$-vector spaces 
\begin{equation}
\label{eq:inertia}
H\!P_\ast([X/G]) \simeq (\bigoplus_{g\in G} H^\ast(X^g,\bbC))^G =: H^\ast_{\mathrm{orb}}(X/\!\!/G,\bbC)\,,
\end{equation}
where the right-hand side stands for orbifold cohomology in the sense of Chen-Ruan \cite{CR}. This is Baranovsky's \cite{Baranovsky} beautiful observation
``periodic cyclic homology $=$ orbifold cohomology''. In Corollary \ref{cor:az4} below, we will extend isomorphism \eqref{eq:inertia} to the case where $X$ is equipped to a $G$-equivariant sheaf of Azumaya algebras. Consult also Example \ref{ex:TP} below for the positive characteristic analogue of \eqref{eq:inertia}.
\end{remark}
%
%
\begin{example}[Topological Hochschild homology]
As explained in \S\ref{sub:THH} below, topological Hochschild homology $THH_\ast(-)$ is a lax symmetric monoidal additive invariant with values in the category of $\bbZ$-graded $k$-vector spaces. Therefore, in the case where $k$ contains the $n^{\mathrm{th}}$ roots of unity, Corollary \ref{cor:main1}(ii) leads to the following isomorphism of $\bbZ$-graded commutative $k$-algebras:
\begin{equation}\label{eq:THH}
THH_\ast([X/G]) \simeq (\bigoplus_{g \in G} THH_\ast(X^g))^G\,.
\end{equation}
To the best of the authors' knowledge, the formula \eqref{eq:THH} is new in the literature. 

Let $l$ be the field obtained from $k$ by adjoining the $n^{\mathrm{th}}$ roots of unity. Note that by combining Proposition \ref{prop:THH} below with the following Morita equivalences
\begin{eqnarray*}
\perf_\dg([X_l/G])\simeq \perf_\dg([X/G])\otimes_k l && \perf_\dg(X_l^g) \simeq \perf_\dg(X^g) \otimes_k l\,,
\end{eqnarray*}
we can remove the assumption that $k$ contains the $n^{\mathrm{th}}$ roots of unity!
\end{example}
\begin{example}[Periodic topological cyclic homology]\label{ex:TP}
Let $k$ be a perfect field of characteristic $p>0$, $W(k)$ the associated ring of $p$-typical Witt vectors, and $K:=W(k)[1/p]$ the fraction field of $W(k)$. As mentioned in Example \ref{ex:TC} below, periodic topological cyclic homology $TP$ is a lax symmetric monoidal additive invariant. Since $TP_0(k)\simeq W(k)$ (see \cite[\S4]{TP}), we observe that $TP(-)_{1/p}$ can be promoted to a lax symmetric monoidal additive invariant with values in $\bbZ/2$-graded $K$-vector spaces. Therefore, in the case where $k$ contains the $n^{\mathrm{th}}$ roots of unity\footnote{Recall that $W(k)$ comes equipped with a multiplicative Teichm\"uller map $k \to W(k)$ and that $K$ is of characteristic zero. This implies that $K$ contains the $n^{\mathrm{th}}$ roots of unity and that $1/n \in K$.} (\eg\ $k=\bbF_{p^d}$ with $d$ the multiplicative order of $p$ modulo $n$), Corollary \ref{cor:main1}(ii) leads to the following isomorphism of $\bbZ/2$-graded commutative $K$-algebras:
\begin{equation}\label{eq:formula-TP}
TP_\ast([X/G])_{1/p} \simeq (\bigoplus_{g \in G} TP_\ast(X^g)_{1/p})^G\,.
\end{equation}
To the best of the authors' knowledge, the formula \eqref{eq:formula-TP} is new in the literature. 

In the case where $X$ is moreover proper, Hesselholt proved in \cite[Thms.~5.1 and 6.8]{TP} that $TP_0(X)\simeq \bigoplus_{i\,\mathrm{even}} H^i_{\mathrm{crys}}(X/W(k))$ and $TP_1(X)\simeq \bigoplus_{i\,\mathrm{odd}} H^i_{\mathrm{crys}}(X/W(k))$, where $H^\ast_{\mathrm{crys}}(-)$ stands for crystalline cohomology. Therefore, in this case, the above formula \eqref{eq:formula-TP} can be re-written as the following isomorphism of (finite dimensional) $\bbZ/2$-graded $K$-vector spaces:
\begin{equation}\label{eq:formula-TP1}
TP_\ast([X/G])_{1/p}\simeq (\bigoplus_{g \in G} H^\ast_{\mathrm{crys}}(X^g))^G\,.
\end{equation}
Morally speaking, \eqref{eq:formula-TP1} is the positive characteristic analogue of \eqref{eq:inertia}. In other words, Baranovsky's beautiful observation admits the following analogue: ``periodic topological cyclic homology = orbifold cohomology in positive characteristic''.
\end{example}
\begin{remark}[Proof of Theorem \ref{thm:main}]
Our  proof of Theorem \ref{thm:main}, and consequently of Corollaries \ref{cor:main1} and \ref{cor:main2}, is different from the proofs of Vistoli \cite{Vistoli} and 
Baranovsky \cite{Baranovsky}
(which are themselves also very different). Nevertheless, we do borrow some of ingredients from Vistoli's proof. 
In fact, using the formalism of
noncommutative motives (see \S\ref{sub:universal}), we are able to ultimately reduce the proof of the formula \eqref{eq:formula-main} to the proof of the $K_0$-case of Vistoli's
formula \eqref{eq:Vistoli}; consult \S\ref{sec:proof} for details. Note, however, that
we {\em cannot} mimic Vistoli's arguments because they depend in an essential way
on the d\'evissage property of $G$-theory (=$K$-theory for smooth schemes), which does not hold for many interesting
additive invariants. For example, as explained by Keller in \cite[Example 1.11]{Exact}, Hochschild homology, and consequently the mixed
  complex ${{\mathsf{C}}}$, do {\em not} satisfy d\'evissage.
  \end{remark}

\begin{remark}[McKay correspondence]
  In many cases, the dg category $\perf_\dg([X/G])$ is known to be Morita equivalent 
to $\perf_\dg(Y)$ for a crepant 
  resolution $Y$ of the (singular) geometric quotient $X/\!\!/G$; see \cite{BezKal2,BKR,KV,
    Kawamata}. This is generally referred to as the ``McKay correspondence''.
Whenever it holds, we can replace $[X/G]$ by $Y$ in all the above formulas. Here is an illustrative example (with $k$ algebraically closed): the cyclic group $G=C_2$ acts on any abelian surface $S$  by the
  involution $a \mapsto -a$ and the Kummer surface   $\mathrm{Km}(S)$ is defined as
the blowup of $S/\!\!/C_2$ in its $16$ singular points.
In this case, the dg category $\perf_\dg([S/C_2])$ is Morita equivalent to 
  $\perf_\dg(\mathrm{Km}(S))$. Consequently, Corollary \ref{cor:main1}(i) leads to an isomorphism 
\begin{equation}\label{eq:McKay}
E(\operatorname{Km}(S))\simeq E(k)^{\oplus 16}\oplus E(S)^{C_2}\,.
\end{equation}
 Since the Kummer surface is Calabi-Yau, the category $\perf(\operatorname{Km}(S))$ does not admit 
any non-trivial semi-orthogonal decompositions. Therefore, the above decomposition \eqref{eq:McKay} is {\em not} induced from a semi-orthogonal decomposition.
\end{remark}
\subsubsection*{Smooth quotients} Let us write $\pi\colon X \to X/\!\!/G$ for the quotient morphism. Our second main result is the following:
\begin{theorem}\label{thm:main2}
\label{thm:smoothquotient}
Let $E\colon \dgcat(k) \to \mathrm{D}$ be an additive invariant $E$ with values in a $\bbZ[1/n]$-linear category. If $X/\!\!/G$ is $k$-smooth (e.g.\ if the $G$-action is free), then the induced morphism $\pi^\ast :E(X/\!\!/G) \r E(X)^G$ is invertible.
\end{theorem}
\begin{example}[Hochschild homology]
As mentioned in Example \ref{ex:mixed}, Hochschild homology $H\!H$ is an additive invariant. Therefore, whenever the geometric quotient $X/\!\!/G$ is $k$-smooth, Theorem \ref{thm:main2} leads to an isomorphism $H\!H(X/\!\!/G) \simeq H\!H(X)^G$. This isomorphism was established in \cite[Prop.~2.1.2]{PolishchukVdB} using the Hochschild-Kostant-Rosenberg theorem.
\end{example}
\begin{remark} \label{rk:natural} Assume that we are in the situation
  of Corollary \ref{cor:main1}(ii) and that all the geometric
  quotients $X^g/\!\!/ C(g)$ are $k$-smooth. In this case,
  \eqref{eq:formula-main3} reduces to an isomorphism\footnote{It is natural to ask if such an isomorphism
    is induced from a semi-orthogonal decomposition of $\perf([X/G])$
    with components $\perf(X^g/\!\!/C(g))$.  This was the motivating
    question for the work \cite{PolishchukVdB}. A positive answer to
    this question was obtained therein in many cases. However, the
    required semi-orthogonal decompositions are usually highly
    non-trivial; this complexity already occurs in the ``simple'' case
    of a symmetric group acting on a product of copies of a
      smooth curve. On the other hand, an example of a (non-faithful) $G$-action where a semi-orthogonal decomposition does not exist was
    constructed in \cite{LuntsBerghSchnurer}.}
  $E([X/G])\simeq \bigoplus_{g\in G\!/\!\sim} E(X^g/\!\!/C(g))$. This holds, for example, in the case of a  symmetric group acting on a product of copies of a smooth curve.
\end{remark}

\subsubsection*{Equivariant Azumaya algebras}
\def\uSpec{\underline{\mathrm{S}}\mathrm{p}\underline{\mathrm{ec}}}

Let $\cF$ be a flat quasi-coherent sheaf of algebras over $[X/G]$, 
and $\perf_\dg([X/G];\cF)$ the canonical dg enhancement of the category of $G$-equivariant perfect $\cF$-modules $\perf([X/G];\cF)$. Given an additive invariant, let us write $E([X/G];\cF)$ instead of $E(\perf_\dg([X/G];\cF))$. Finally, given a cyclic subgroup $\sigma \in \varphi$, let us denote by $\cF_\sigma$ the pull-back of $\cF$ along the morphism $[X^\sigma/\sigma] \to [X/G]$; note that $\cF_\sigma$ is a $N(\sigma)$-equivariant sheaf of algebras over $X^\sigma$. Under the above notations, our third main result which extends Theorem \ref{thm:main}, is the following: 
%
%
\begin{theorem}
\label{thm:formula-main4}
For every additive invariant $E\colon \dgcat(k) \to \mathrm{D}$, with values in a 
$\bbZ[1/n]$-linear category, we have an isomorphism
\begin{equation}
\label{eq:algebras}
E([X/G]; \cF) \simeq \bigoplus_{\sigma \in \vphi} \widetilde{E}([X^\sigma/\sigma]; \cF_\sigma)^{N(\sigma)}
\end{equation}
induced by pull-back with respect to the morphisms $[X^\sigma/\sigma]\r[X/G]$.
\end{theorem}
From now on we will assume that $\cF$ is a sheaf of Azumaya algebras over $[X/G]$, \ie a $G$-equivariant sheaf of Azumaya algebras over $X$. We will write $r$ for the product of the ranks of $\cF$ (at each one of the connected components of $X$). Let us denote by $\cF_\sigma \# \sigma$ the sheaf of skew group
algebras corresponding to the $\sigma$-action on $\cF_\sigma$. Note
that since $\sigma$ acts trivially on $X^\sigma$,
$\cF_\sigma \# \sigma$ is a sheaf of $\cO_{X^\sigma}$-algebras. Finally, let us write 
$\cZ_{\sigma}$ for the center of $\cF_\sigma\#\sigma$ and
$Y_\sigma:=\uSpec(\cZ_\sigma)$. By construction, $\cF_\sigma\#\sigma$
(and hence $\cZ_{\sigma}$) is
$\sigma$-graded. 

Let $E\colon \dgcat(k) \to \mathrm{D}$ be an additive invariant as in
Theorem \ref{thm:formula-main4}, and assume that $k$ contains the
$n^{\mathrm{th}}$ roots of unity. As mentioned above, $\cZ_\sigma$ is $\sigma$-graded and therefore it comes equipped with a canonical $\sigma^\vee$-action, where $\sigma^\vee:=\Hom(\sigma, k^\times)$ stands for the dual cyclic group.
Hence, by functoriality,
$E(Y_\sigma)=E(X^\sigma;\cZ_\sigma)$ inherits a
$\bbZ[\sigma^\vee]_{1/n}$-action. Since $k$ contains the
$n^{\mathrm{th}}$ roots of unity, we have a character isomorphism
$R(\sigma)\simeq \bbZ[\sigma^\vee]$. Therefore, we can consider the
direct summand $e_\sigma E(Y_\sigma)$ of $E(Y_\sigma)$ associated to
the idempotent $e_\sigma \in R(\sigma)_{1/n}$. In the case
where the category $\mathrm{D}$ is $l$-linear for a field $l$ which
contains the $n^{\mathrm{th}}$ roots of unity and $1/n\in l$, the
$\sigma^\vee$-action on $E(Y_\sigma)$ may be translated
back\footnote{By choosing an isomorphism $\epsilon$ between the
  $n^{\mathrm{th}}$ roots of unity in $k$ and in $l$, we can identify
  $\sigma^{\vee\vee}:=\Hom(\Hom(\sigma, k^\times),l^\times)$ with
  $\sigma$. This identification is moreover natural on $\sigma$.}
into a $(\sigma^{\vee\vee}=\sigma)$-grading on $E(Y_\sigma)$. Whenever
an object $\cO$ is equipped with a $\langle g\rangle$-grading, we will
write $\cO_g$ for its degree $g$ part. In particular, $E(Y_g)_g$
stands for the degree $g$ part of $E(Y_{\langle g\rangle})$. Under the
above notations, Theorem \ref{thm:formula-main4} admits the following
refinements:
\begin{corollary}\label{cor:az1} 
Assume that $\cF$ is a sheaf of Azumaya algebras over $[X/G]$. Under this assumption, the following holds:
\begin{itemize}
\item[(i)] The structural morphism $Y_\sigma \to X^\sigma$ is a $N(\sigma)$-equivariant $\sigma^\vee$-Galois cover. Moreover, for every $g \in G$, the sheaf $\cL_g:=(\cZ_{\langle g \rangle})_g$ is a $C(g)$-equivariant line bundle (equipped with a $C(g)$-equivariant flat connection) on $X^g:=X^{\langle g\rangle}$; when $k=\bbC$, we will denote by $L_g$ the associated $C(g)$-equivariant rank one local system on $X^g$. 
\item[(ii)]  If $k$ contains the $n^{\mathrm{th}}$ roots of unity and the category $\mathrm{D}$ is $\bbZ[1/nr]$-linear, then \eqref{eq:algebras} reduces to an isomorphism
\begin{equation}
\label{eq:algebras3}
E([X/G]; \cF)\simeq \bigoplus_{\sigma \in \vphi} (e_\sigma E(Y_\sigma))^{N(\sigma)}\,.
\end{equation}
\item[(iii)] If $k$ contains the $n^{\mathrm{th}}$ roots of unity and $\mathrm{D}$ is $l$-linear for a
field $l$ which contains the $n^{\mathrm{th}}$ roots of unity and $1/nr \in l$, then \eqref{eq:algebras3} reduces to an isomorphism 
\begin{equation*}
E([X/G]; \cF)\simeq \bigoplus_{g\in G\!/\!\sim} E(Y_g)_g^{C(g)} \simeq (\bigoplus_{g\in G} E(Y_g)_g)^G\,.
\end{equation*}
\end{itemize}
\end{corollary}
\begin{example}[Algebraic $K$-theory]\label{ex:K-theory2}
Algebraic $K$-theory is an additive invariant. Therefore, in the case where $k$ contains the $n^{\mathrm{th}}$ roots of unity, Corollary \ref{cor:az1}(iii) leads to the following isomorphism of $\bbZ$-graded $\bbC$-vector spaces
\begin{equation}\label{eq:twisted}
K_\ast([X/G]; \cF)_\bbC \simeq \bigoplus_{g\in G\!/\!\sim} K_\ast(Y_g)_{\bbC, g}^{C(g)}\,.
\end{equation}
It may be shown that in the particular
case when $k=\bbC$ and $\cF$ is induced from a cohomology class $\alpha \in H^2(G,\bbC^\times)$ by pull-back along the morphism $[X/G] \to BG$, the formula \eqref{eq:twisted} reduces to the algebraic analogue of the formula established by Adem-Ruan \cite[\S7]{AR} in the topological setting of twisted orbifold 
$K$-theory.
\end{example}

%
%
%
\begin{example}[Mixed complex]\label{ex:mixed3}
The mixed complex ${{\mathsf{C}}}$ is an additive invariant. Moreover, by construction, it sends graded dg categories to graded mixed complexes. This implies that $\mathsf{C}(Y_g)=\mathsf{C}(X^g;\cZ_g)$ comes equipped with a canonical $\langle g\rangle$-grading, which clearly agrees with the one defined above. Therefore, in the case where $k$ contains the $n^{\mathrm{th}}$ roots of unity and $1/nr \in k$, Corollary \ref{cor:az1}(iii) leads to the following isomorphism in the derived category of mixed complexes
\begin{equation}
\label{eq:formula8}
\mathsf{C}([X/G]; \cF)\simeq (\bigoplus_{g\in G} \mathsf{C}(Y_g)_g)^G\,,
\end{equation}
where $\mathsf{C}(Y_g)_g$ stands for the degree $g$ part of the tautological $\langle g \rangle$-grading on $\mathsf{C}(Y_g)$. The morphism \eqref{eq:formula8} is given by first restricting $\mathsf{C}([X/G]; \cF)$ to $\mathsf{C}([X^g/\langle g\rangle]; \cF_{\langle g\rangle})\simeq \mathsf{C}(Y_g)$ and then by projecting $\mathsf{C}(Y_g)$ onto its degree $g$ part $\mathsf{C}(Y_g)_g$. This shows, in particular, that \eqref{eq:formula8} is defined over $k$. Since $\mathsf{C}$ is compatible with base-change, we can then remove the assumption that $k$ contains the $n^{\mathrm{th}}$ roots of unity!
\end{example}
A formula similar to \eqref{eq:formula8} holds for all the variants of
cyclic homology. In the particular case of Hochschild homology $H\!H$
and periodic cyclic homology $H\!P$, this formula admits the following
(geometric) refinements:
\begin{corollary}
\label{cor:az3}
Assume that $1/nr\in k$. In the case of $H\!H_\ast$, \eqref{eq:formula8} reduces to an isomorphism of $\bbZ$-graded $k$-vector spaces:
\begin{equation}
\label{eq:formula6}
H\!H_\ast([X/G];\cF)\simeq (\bigoplus_{g\in G}H\!H_{\ast}(X^g,\cL_g))^G\,.
\end{equation}
\end{corollary}
\begin{proof}
Corollary \ref{cor:az1}(i) implies that $H\!H_\ast(Y_g)_g \simeq H\!H_\ast(X^g; \cZ_g)_g \simeq H\!H_\ast(X^g, \cL_g)$. By combining these identifications with \eqref{eq:formula8}, we then obtain \eqref{eq:formula6}.
\end{proof}
In the
particular 
case when $k=\bbC$ and $\cF$ is induced from a cohomology class
$\alpha \in H^2(G,\bbC^\times)$ by pull-back along the morphism
$[X/G] \to BG$, the formula \eqref{eq:formula6} was established by
Baranovsky and Petrov in \cite[Thm.~6]{Baranovsky-Petrov}. In the case
where $k$ is of characteristic $0$, the formula \eqref{eq:formula6}
was announced by C\v{a}ld\v{a}raru in \cite{Caldararu3}. Therein,
C\v{a}ld\v{a}raru and Arinkin used it in order to conceptually explain
an {\em ad hoc} computation of Vafa and Witten \cite{VafaWitten}
concerning~elliptic~curves.


\begin{corollary}[Twisted orbifold cohomology]\label{cor:az4}
Assume that $k=\bbC$. In the case of $H\!P_\ast$, \eqref{eq:formula8} reduces to an isomorphism\footnote{The data $L:=(L_g)_g$ defines a $G$-equivariant local system on $\amalg_{g \in G} X^g$ or, equivalently, a local system on the inertia stack $I(\cX)$ of $\cX:=[X/G]$. Under these notations, the above isomorphism \eqref{eq:inertia1} can be re-written as $H\!P_\ast(\cX;\cF) \simeq H^\ast(I(\cX),L)$.} of $\bbZ/2$-graded $\bbC$-vector spaces:
\begin{equation}
\label{eq:inertia1}
H\!P_\ast([X/G];\cF)\simeq (\bigoplus_{g\in G} H^\ast(X^g,L_g))^G\,.
\end{equation}
\end{corollary}
\begin{proof}
Let us write $f\colon Y_g \to X^g$ for the structural morphism. Thanks to the Hochschild-Kostant-Rosenberg theorem (see \cite{FeiginTsygan,Exact,CRSS}), the equality $f_\ast(\underline{\bbC})= \bigoplus_{g \in G} L_g$, and Corollary \ref{cor:az1}(i), we have the following isomorphisms: 
$$ H\!P_\ast(Y_g)_g \simeq H^\ast(Y_g, \bbC)_g \simeq H^\ast(X^g, f_\ast(\underline{\bbC}))_g\simeq H^\ast(X^g, L_g)\,.$$
By combining them with \eqref{eq:formula8}, we then obtain \eqref{eq:inertia1}.
\end{proof}
\subsection*{Notations}
Throughout the article, $k$ will be a field of characteristic $p\geq 0$ and $G$ a finite group of order $n$. Except in Appendix \ref{sec:appendix}, we will always assume that $1/n \in k$. In order to simplify the exposition, we will often write $\bullet$ instead of $\mathrm{Spec}(k)$, $(-)_{1/n}$ instead of $(-)_{\bbZ[1/n]}$, and make no notational distinction between a dg functor and its image under an additive invariant.
%
%

\medbreak\noindent\textbf{Acknowledgments:}
The authors are grateful to Andrei C\v{a}ld\v{a}raru for several interesting discussions; in particular, those regarding Corollary \ref{cor:az3}. The first author also would like to thank Lars Hesselholt, Michael Hopkins and Akhil Mathew for useful discussions concerning topological Hochschild homology and periodic topological cyclic homology. The second author is also grateful to Valery Lunts and Alexander Polishchuk for their insights on the relation between motivic and semi-orthogonal decompositions
of orbifolds. Finally, the authors also would like to thank Nick Kuhn for comments on Example \ref{ex:algebraic1} and for the references \cite{Dieck,Kuhn}.
\section{Preliminaries}

\subsection{Dg categories}\label{sub:dg}
Let $(\cC(k),\otimes, k)$ be the category of (cochain) complexes of 
$k$-vector spaces; we use cohomological notation. A {\em differential
  graded (=dg) category $\cA$} is a category enriched over $\cC(k)$
and a {\em dg functor} $F\colon\cA\to \cB$ is a functor enriched over
$\cC(k)$; consult Keller's ICM survey
\cite{ICM-Keller}. Recall from \S\ref{sec:intro} that $\dgcat(k)$ stands for the category of (essentially small) dg categories and dg functors.

Let $\cA$ be a dg category. The opposite dg category $\cA^\op$ has the
same objects and $\cA^\op(x,y):=\cA(y,x)$. The category $\dgHo(\cA)$ has the same objects as $\cA$ and morphisms $\dgHo(\cA)(x,y):=H^0\cA(x,y)$, where $H^0(-)$ stands for the $0^{\mathrm{th}}$-cohomology functor. A {\em right dg
  $\cA$-module} is a dg functor $M\colon \cA^\op \to \cC_\dg(k)$ with values
in the dg category $\cC_\dg(k)$ of complexes of $k$-vector spaces. Let
us write $\cC(\cA)$ for the category of right dg
$\cA$-modules. Following \cite[\S3.2]{ICM-Keller}, the derived
category $\cD(\cA)$ of $\cA$ is defined as the localization of
$\cC(\cA)$ with respect to the objectwise quasi-isomorphisms. Let
$\cD_c(\cA)$ be the triangulated subcategory of compact objects.

A dg functor $F\colon\cA\to \cB$ is called a {\em Morita equivalence} if the restriction functor along $F$
induces an equivalence on derived categories $\cD(\cB) \stackrel{\simeq}{\to}
\cD(\cA)$; see \cite[\S4.6]{ICM-Keller}. As explained in
\cite[\S1.6]{book}, the category $\dgcat(k)$ admits a Quillen model
structure whose weak equivalences are the Morita equivalences. Let us
denote by $\Hmo(k)$ the associated homotopy category.

The {\em tensor product $\cA\otimes\cB$} of dg categories is defined
as follows: the set of objects is the cartesian product of the sets of objects of $\cA$ and $\cB$ and
$(\cA\otimes\cB)((x,w),(y,z)):= \cA(x,y) \otimes \cB(w,z)$. As
explained in \cite[\S2.3]{ICM-Keller}, this construction gives rise to
a symmetric monoidal structure on $\dgcat(k)$, which descends to
$\Hmo(k)$. 

A {\em dg $\cA\text{-}\cB$-bimodule} is a dg functor
$\mathrm{B}\colon \cA\otimes \cB^\op \to \cC_\dg(k)$ or, equivalently, a
right dg $(\cA^\op \otimes \cB)$-module. A standard example is the dg
$\cA\text{-}\cB$-bimodule
\begin{eqnarray}\label{eq:bimodule2}
{}_F\mathrm{B}:\cA\otimes \cB^\op \too \cC_\dg(k) && (x,z) \mapsto \cB(z,F(x))
\end{eqnarray}
associated to a dg functor $F:\cA\to \cB$. Let us write $\rep(\cA,\cB)$ for the full triangulated subcategory of $\cD(\cA^\op \otimes \cB)$ consisting of those dg $\cA\text{-}\cB$-modules $\mathrm{B}$ such that for every object $x \in \cA$ the associated right dg $\cB$-module $\mathrm{B}(x,-)$ belongs to $\cD_c(\cB)$.

Following Kontsevich \cite{Miami,finMot,IAS}, a dg category $\cA$ is called {\em smooth} if the dg $\cA\text{-}\cA$-bimodule ${}_{\id}\mathrm{B}$ belongs to the triangulated category $\cD_c(\cA^\op\otimes \cA)$ and {\em proper} if $\sum_i \mathrm{dim}\, H^i\cA(x,y)< \infty$ for any ordered pair of objects $(x,y)$. 
\subsection{Additive invariants}
\label{sec:additive}
Recall from Bondal-Orlov \cite[Def.~2.4]{BO}\cite{BO1} that a {\em semi-orthogonal decomposition} of a triangulated category $\cT$, denoted by $\cT=\langle \cT_1, \cT_2\rangle$, consists of full triangulated subcategories $\cT_1, \cT_2 \subseteq \cT$ satisfying the following conditions: the inclusions $\cT_1, \cT_2 \subseteq \cT$ admit left and right adjoints; the triangulated category $\cT$ is generated by the objects of $\cT_1$ and $\cT_2$; and $\Hom_\cT(\cT_2, \cT_1)=0$. A functor $E\colon \dgcat(k)
\to \mathrm{D}$, with values in an idempotent complete additive category, is called an
{\em additive invariant} if it satisfies the following two conditions:
\begin{itemize}
\item[(i)] It sends the Morita equivalences (see \S\ref{sub:dg}) to isomorphisms.
\item[(ii)] Given dg categories $\cA,\cC \subseteq \cB$ such that $\dgHo(\cB)=\langle\dgHo(\cA), \dgHo(\cC) \rangle$, the inclusions $\cA, \cC\subseteq \cB$ induce an isomorphism $E(\cA) \oplus E(\cC) \simeq E(\cB)$.
\end{itemize}
\begin{example}[Algebraic $K$-theory]\label{ex:K-theory}
Algebraic $K$-theory gives rise to a (lax symmetric monoidal) additive invariant $K\colon \dgcat(k) \to \Ho(\Spt)$ with values in the homotopy category of spectra; see \cite[\S2.2.1]{book}. Classical variants such as mod-$l^\nu$ algebraic $K$-theory $K(-;\bbZ/l^\nu)$, Karoubi-Villamayor $K$-theory $KV$, nonconnective algebraic $K$-theory $\bbK$, homotopy $K$-theory $KH$, and \'etale $K$-theory $K^{\mathrm{et}}(-;\bbZ/l^\nu)$, also give rise to additive invariants; consult \cite[\S2.2.2-\S2.2.6]{book} for details.
\end{example}
\begin{example}[Cyclic homology]\label{ex:cyclic}
Cyclic homology gives rise to an additive invariant $H\!C\colon \dgcat(k) \to \cD(k)$ with values in the derived category of the base field $k$; see \cite[\S2.2.9]{book}. Classical variants such as Hochschild homology $H\!H$, periodic cyclic homology $H\!P$, and negative cyclic homology $H\!C^-$, also give rise to additive invariants; consult \cite[\S2.2.8-\S2.2.11]{book} for details.
\end{example}
\begin{example}[Mixed complex]\label{ex:mixed}
  A {\em mixed complex} is a (right) dg module over the algebra of
  dual numbers $\Lambda:=k[\epsilon]/\epsilon^2$ with
  $\mathrm{deg}(\epsilon)=-1$ and $d(\epsilon)=0$. Note that $\Lambda$ is a cocommutative graded dg-Hopf algebra with $\Delta(\epsilon)=\epsilon\otimes 1+1\otimes \epsilon$. Therefore, the tensor product $-\otimes_k-$ makes the derived category $\cD(\Lambda)$
into a symmetric monoidal category. The mixed complex gives rise to a (symmetric monoidal) additive invariant
  ${{\mathsf{C}}}\colon \dgcat(k) \to \cD(\Lambda)$. Moreover, all the additive
  invariants of Example \ref{ex:cyclic} factor through ${{\mathsf{C}}}$; consult
  \cite[\S2.2.7]{book} for details. 
\end{example}
\begin{example}[Topological Hochschild homology]\label{ex:TC}
Topological Hochschild homology gives rise to a (lax symmetric monoidal) additive invariant $TH\!H\colon \dgcat(k) \to \Ho(\Spt)$; see \cite[\S2.2.12]{book}. Variants such as topological cyclic homology $TC$ (see \cite[\S2.2.13]{book}) and periodic topological cyclic homology $TP$ (see \cite{TP}) also give rise to (lax symmetric monoidal) additive invariants.
\end{example}
\subsection{Change of coefficients}\label{sub:change}
Given a ring homomorphism $S\r R$ and a $S$-linear idempotent complete additive category 
$\mathrm{D}$, let $\mathrm{D}_R$ be the idempotent completion~of the category obtained by tensoring the Hom-sets of $\mathrm{D}$ with $R$. This procedure leads to a 
$2$-functor from the $2$-category
of $S$-linear idempotent complete additive categories to the $2$-category of $R$-linear idempotent complete additive categories. Moreover, we have an obvious ``change of coefficients'' functor $(-)_R\colon \mathrm{D} \to \mathrm{D}_R$. 
%
\subsection{Noncommutative motives}\label{sub:universal}
  Given dg categories $\cA$ and $\cB$, there is a
  natural bijection between $\Hom_{\Hmo(k)}(\cA,\cB)$ and the set of
  isomorphism classes of the category $\rep(\cA,\cB)$. Under this bijection, the composition law of $\Hmo(k)$ corresponds to the tensor product of bimodules. Therefore, since the dg $\cA\text{-}\cB$-bimodules
  \eqref{eq:bimodule2} belong to $\rep(\cA,\cB)$, we have the symmetric monoidal functor:
\begin{eqnarray}\label{eq:functor1}
\dgcat(k)\too \Hmo(k) & \cA \mapsto \cA & (\cA\stackrel{F}{\to} \cB) \mapsto {}_F \mathrm{B}\,.
\end{eqnarray}
The {\em additivization} of $\Hmo(k)$ is the additive category
$\Hmo_0(k)$ with the same objects as $\Hmo(k)$ and with abelian groups of morphisms $\Hom_{\Hmo_0(k)}(\cA,\cB)$ given by the Grothendieck group
$K_0\rep(\cA,\cB)$ of the triangulated category
$\rep(\cA,\cB)$. The category of {\em noncommutative motives $\NMot(k)$} is defined as the idempotent completion of $\Hmo_0(k)$ and the {\em universal additive invariant} 
$$\U(-):\dgcat(k)\too \NMot(k)
$$
as the composition of \eqref{eq:functor1} with the canonical functors $\Hmo(k)\to \Hmo_0(k)$ and $\Hmo_0(k)\to \NMot(k)$. Given a commutative ring of coefficients $R$, the category $\NMot(k)_R$ is defined as in \S\ref{sub:change}. 
As explained in
  \cite[\S2.3]{book}, given any $R$-linear 
  idempotent complete additive ((symmetric) monoidal)
  category $\mathrm{D}$, pre-composition with the functor $\U(-)_R$
  gives rise to induced equivalences of categories
\begin{eqnarray}
\Fun_{R\text{-}\mathrm{linear}}(\NMot(k)_R,\mathrm{D}) \stackrel{\simeq}{\too} \Fun_{\mathrm{add}}(\dgcat(k),\mathrm{D}) \label{eq:equivalence1} \\
\Fun^\otimes_{R\text{-}\mathrm{linear}}(\NMot(k)_R,\mathrm{D}) \stackrel{\simeq}{\too} \Fun^\otimes_{\mathrm{add}}(\dgcat(k),\mathrm{D})\label{eq:equivalence2}
\end{eqnarray}
where the left-hand side stands for the category of $R$-linear ((lax) 
(symmetric) monoidal) functors and the right-hand side for the category of
((lax) (symmetric) monoidal)~additive~invariants. For further information on noncommutative motives we invite the reader to consult the recent book \cite{book} and survey \cite{survey}.
\begin{remark}\label{rk:Homs}
Given dg categories $\cA$ and $\cB$, with $\cA$ smooth and proper, we have an equivalence $\rep(\cA,\cB) \simeq \cD_c(\cA^\op \otimes \cB)$. Consequently, we obtain isomorphisms
\begin{equation*}
\Hom_{\NMot(k)}(\U(\cA),\U(\cB)) := K_0(\rep(\cA,\cB))\simeq K_0(\cA^\op \otimes \cB)\,.
\end{equation*}
\end{remark}
\subsection{Perfect complexes of trivial $G$-actions}
Let $X$ be a quasi-compact quasi-separated $k$-scheme equipped with a trivial $G$-action. In this case, we have $[X/G]=X \times_\bullet BG$. This leads naturally to the following dg functor:
\begin{eqnarray}
\label{eq:trivial}
\perf_{\dg}(X) \otimes \perf_{\dg}(BG)\too \perf_{\dg}([X/G]) && (M, V)\mapsto M\boxtimes V\,.
\end{eqnarray}
\begin{proposition}\label{prop:generator}
The above dg functor \eqref{eq:trivial} is a Morita equivalence\footnote{The assumption $1/n\in k$ might be unnecessary for Proposition \ref{prop:generator}. Nevertheless, it is necessary to guaranty that the dg categories $\perf_\dg(BG)$ and $k[G]$ are Morita equivalent.}.
\end{proposition}
\begin{proof}
Since $1/n \in k$, $\perf_\dg(BG)$ is Morita equivalent to $\perf_\dg(k[G])$ and, consequently, to $k[G]$. Therefore, it suffices to show that the induced functor
$$
\perf_\dg(X)\otimes k[G] \too \perf_\dg([X/G])
$$
is a Morita equivalence. Let us denote by $\cD_{{\operatorname{Qch}}}(X)$, resp. $\cD_{{\operatorname{Qch}}}([X/G])$, the full triangulated subcategory of $\cD(X)$, resp. $\cD([X/G])$, consisting of those complexes of $\cO_X$-modules, resp. $G$-equivariant complexes of $\cO_X$-modules, with quasi-coherent cohomology. Thanks to Neeman's
celebrated result \cite[Thm.~2.1]{Neeman}, the full subcategory of compact objects of $\cD_{{\operatorname{Qch}}}(X)$ agrees with $\perf(X)$. Under the assumption $1/n\in k$, a similar result holds for the (global) orbifold $[X/G]$; see \cite[Thm.~C]{HallRydh}. 
As proved in \cite[Thm.~3.1.1]{BV}, the triangulated category
  $\cD_{\operatorname{Qch}}(X)$ admits a compact generator
  $\cG$. Consequently, since $G$ acts trivially on $X$,
  $\oplus_{g \in G} \cG$ is a compact generator of $\mathcal{D}_{\operatorname{Qch}}([X/G])$. By passing to compact objects, we then obtain the following Morita equivalences
%
%
%
%
%
%
%
\begin{eqnarray*}
\perf_\dg([X/G]) \simeq \perf_\dg({\bf R}\End(\oplus_{g \in G}\cG)) && \perf_\dg(X) \simeq \perf_\dg({\bf R}\End(\cG))\,,
\end{eqnarray*}
where ${\bf R}\End(-)$ stands for the (derived) dg $k$-algebra of endomorphisms. The proof follows now from the canonical quasi-isomorphism of dg $k$-algebras between ${\bf R}\End(\oplus_{g \in G} \cG)$ and ${\bf R} \End(\cG) \otimes k[G]$ 
\end{proof}
\begin{corollary}\label{cor:monoids}
We have an isomorphism of monoids $\U(X) \otimes \U(BG) \simeq \U([X/G])$.
\end{corollary}
\begin{proof}
Combine Proposition \ref{prop:generator} with the fact that $\U$ is symmetric monoidal.
\end{proof}
\subsection{Galois descent for representation rings}
Let $l/k$ be a finite Galois field extension with Galois group $\Gamma$. Consider the induced homomorphism
\begin{equation}\label{eq:induced}
-\otimes_k l \colon R(G)_{1/n} \too R_l(G)_{1/n}^\Gamma\,,
\end{equation}
where $R_l(G)$ stands for the representation ring of $G$ over $l$.
\begin{proposition} \label{prop:galois}
The above homomorphism \eqref{eq:induced} is invertible.
\end{proposition}
\begin{proof} 
Since $1/n \in k$, the group algebra $k[G]$ is semi-simple. Hence, we have $k[G]\simeq \oplus_{i=1}^m A_i$ with $A_i$ a central simple algebra over its center $l_i$. As explained in\footnote{The result in {\em loc. cit.} is stated in characteristic $0$. However, it is well-known that if $p\nmid |G|$, then the blocks in characteristic $p$ are obtained by 
reduction from those in characteristic $0$.} \cite[Thm.~(33.7)]{CurtisReiner}, the index of $A_i$ divides $n$. Consequently, the functor $A_i \otimes_{l_i}-$ gives rise to isomorphisms $K_0(l_i)_{1/n}\simeq K_0(A_i)_{1/n}$ and $K_0(l_i\otimes_k l)_{1/n}\simeq K_0(A_i\otimes_k l)_{1/n}$. Therefore, it suffices to show that the homomorphism~$-\otimes_k l \colon K_0(l_i) \to K_0(l_i \otimes_k~l)^\Gamma$ is invertible. Note that $l_i \otimes_k l$ decomposes into a direct sum of copies of a field. Moreover, $\Gamma$ permutes these copies. This implies that $K_0(l_i \otimes_k l)\simeq \bbZ\{S\}$, where $S$ is a set on which $\Gamma$ acts transitively. The proof follows now from the fact that the diagonal inclusion $\bbZ \to \bbZ\{S\}$ induces an isomorphism $\bbZ \simeq (\bbZ\{S\})^\Gamma$.
\end{proof}
\subsection{Galois descent for topological Hochschild homology}\label{sub:THH}
As mentioned in Example \ref{ex:TC}, topological Hochschild homology $THH$ is a lax symmetric monoidal additive invariant. Since the ``inclusion of the $0^{\mathrm{th}}$ skeleton'' yields a ring isomorphism $k \stackrel{\sim}{\to} THH_0(k)$, we observe that $THH_\ast(-)$ can be promoted to a lax symmetric monoidal additive invariant with values in $\bbZ$-graded $k$-vector spaces.

Let $l/k$ be a finite Galois field extension with Galois group $\Gamma$. Given a dg category $\cA$, consider the induced homomorphism of $\bbZ$-graded $k$-vector spaces
\begin{equation}\label{eq:induced-last}
-\otimes_k l \colon THH_\ast(\cA) \too THH_\ast(\cA\otimes_k l)^\Gamma\,.
\end{equation}
\begin{proposition}\label{prop:THH}
The above homomorphism \eqref{eq:induced-last} is invertible.
\end{proposition}
\begin{proof}
Note first that by definition of topological Hochschild homology, the following two procedures lead to the same $\bbZ$-graded $k$-vector spaces:
\begin{itemize}
\item[(i)] First apply the functor $THH_\ast(-)$ to the dg $l$-linear category $\cA\otimes_k l$, and then consider $THH_\ast(\cA\otimes_k l)$ as a $\bbZ$-graded $k$-vector space.
\item[(ii)] First consider $\cA\otimes_kl$ as a dg $k$-linear category, and then apply $THH_\ast(-)$.
\end{itemize}
Consider $\cA\otimes_k l$ as a dg $k$-linear category. Consider also the dg $(\cA\otimes_k l)\text{-}\cA$-bimodule:
\begin{eqnarray}\label{eq:bimodule}
(\cA\otimes_k l) \otimes \cA^\op \too \cC_\dg(k) && (z,x) \mapsto \cA\otimes_k l ((-\otimes_k l)(x),z)\,. 
\end{eqnarray}
Since the field extension $l/k$ is finite, \eqref{eq:bimodule} belongs to the category $\rep(\cA\otimes_k l, \cA)$. Consequently, \eqref{eq:bimodule} corresponds to a morphism $\cA\otimes_k l \to \cA$ in the homotopy category $\Hmo(k)$. Using the equivalence of categories \eqref{eq:equivalence1}, we then obtain an induced homomorphism of $\bbZ$-graded $k$-vector spaces $\mathrm{res}\colon THH_\ast(\cA\otimes_k l) \to THH_\ast(\cA)$. In the particular case where $\cA=k$, the homomorphism $\mathrm{res}\colon THH_0(l) \to THH_0(k)$ agrees with the field trace homomorphism $\mathrm{tr}\colon l \to k$. Note that $\mathrm{tr}$ is surjective.

We now have all the ingredients necessary to conclude the proof. Choose an element $\lambda\in l$ such that $\mathrm{tr}(\lambda)=1$, and consider the following composition:
\begin{equation}\label{eq:composed}
THH_\ast(\cA\otimes_k l)^\Gamma \subset THH_\ast(\cA\otimes_k l) \stackrel{\lambda \cdot -}{\too} THH_\ast(\cA\otimes_k l) \stackrel{\mathrm{res}}{\too} THH_\ast(\cA)\,.
\end{equation}
We claim that the homomorphisms \eqref{eq:induced-last} and \eqref{eq:composed} are inverse to each other. On the one hand, given $\alpha \in THH_\ast(\cA)$, we have the following equalities
$$ \mathrm{res}(\lambda \cdot (\alpha \otimes_k l)) \stackrel{(a)}{=} \mathrm{tr}(\lambda) \cdot \alpha = 1 \cdot \alpha = \alpha\,,$$
where $(a)$ follows from the projection formula. On the other hand, given $\beta \in THH_\ast(\cA\otimes_k l)^\Gamma$, we have the following equalities
$$ \mathrm{res}(\lambda \cdot \beta) \otimes_k l \stackrel{(b)}{=} \sum_{\gamma \in \Gamma} \gamma(\lambda \cdot \beta) = \sum_{\gamma \in \Gamma} \gamma^{-1}(\lambda) \cdot \gamma(\beta)  =  (\sum_{\gamma \in \Gamma} \gamma^{-1}(\lambda)) \cdot \beta \stackrel{(c)}{=} (\mathrm{tr}(\lambda)\otimes_k l)\cdot \beta =\beta\,,$$
where $(b)$ follows from \cite[Prop.~5.5]{Descent} and $(c)$ from the fact that $l/k$ is a finite Galois field extension. This finishes the proof.
\end{proof}
\section{$R(G)$-action on additive invariants}\label{sec:action}
We start with some generalities. Let $(\mathrm{D}, \otimes, {\bf 1})$ be a $\bbZ$-linear monoidal category. Given an object $\cO\in \mathrm{D}$, consider the associated abelian group $\Hom({\bf 1},\cO)$.
\begin{itemize}
\item[(i)] If $\cO$ is a monoid in $\mathrm{D}$, then $\Hom({\bf 1}, \cO)$ is a ring. Similarly, if $\mathrm{D}$ is a symmetric monoidal category and $\cO$ is a commutative monoid, then $\Hom({\bf 1},\cO)$ is also a commutative monoid.
\item[(ii)] If $\cO$ is a monoid in $\mathrm{D}$, then we have an induced ring homomorphism
\begin{eqnarray*}
\Hom({\bf 1},\cO) \too \Hom(\cO,\cO) && f \mapsto \phi(f):=m \circ (f \otimes \id_\cO)\,,
\end{eqnarray*}
where $m\colon \cO \otimes \cO \to \cO$ stands for the multiplication map. In other words, the ring $\Hom({\bf 1},\cO)$ acts on the object $\cO$.
\item[(iii)] If $\mathrm{D}$ is a symmetric monoidal category and $\cO$ is a commutative monoid, then the monoid structure of $\cO$ is $\Hom({\bf 1},\cO)$-linear. In other words, the equality $\phi(fg)\circ m = m \circ (\phi(g) \otimes \phi(g))$ holds for every $f, g \in \Hom({\bf 1},\cO)$.
%
%
\end{itemize}
Let us now consider the case where $\mathrm{D}=\NMot(k)$. Given a quasi-compact quasi-separated $k$-scheme $X$ (or, more generally, a suitable algebraic stack $\cX$), the tensor product $-\otimes_X-$ makes $\perf_\dg(X)$ into a commutative monoid in $\dgcat(k)$. Since the universal additive invariant is symmetric monoidal, we obtain a commutative monoid $\U(X)$ in $\NMot(k)$. The above general considerations, combined with the following isomorphism (see Remark \ref{rk:Homs})
$$ \Hom_{\NMot(k)}(\U(k),\U(X))\simeq K_0(X)\,,$$
allow us then to conclude that the Grothendieck ring $K_0(X)$ acts on $\U(X)$. Moreover, the monoid structure of $\U(X)$ is $K_0(X)$-linear. Note that under the notations of \S\ref{sub:dg}, the $K_0(X)$-action on $\U(X)$ may be explicitly described as follows:
\begin{eqnarray*}
K_0(\perf(X))\too K_0(\rep(\perf_{\dg}(X),\perf_{\dg}(X))) && [M]\mapsto [{}_{(M\otimes_X -)} \mathrm{B}]\,.
\end{eqnarray*}
Given an additive invariant $E$, the equivalence of categories \eqref{eq:equivalence1} implies, by functoriality, that the Grothendieck ring $K_0(X)$ acts on $E(X)$. If $E$ is moreover symmetric monoidal, then the monoid structure of $E(X)$ is $K_0(X)$-linear.

Finally, let us consider the case where $\cX=[X/G]$. In this case, the morphism $[X/G]\to BG$ induces a ring homomorphism $R(G) \to K_0([X/G])$. By combining it with the above considerations, we obtain a canonical $R(G)$-action on $E([X/G])$ for every additive invariant $E$, which may be explicitly described~as~follows:
\begin{eqnarray*}
R(G) \too \End_\mathrm{D}(E([X/G])) && [V]\mapsto (E([X/G]) \stackrel{V\otimes_k -}{\too} E([X/G]))\,.
\end{eqnarray*}
\section{Primitive part of $R(\sigma)$}
In this section, given $\sigma \in \varphi$, we introduce the primitive part of the representation ring $R(\sigma)$. We start with some preliminaries on cycle group rings.

\subsubsection{Cyclic group rings}
Let $\Sigma$ be a commutative ring and $\rho$ a cyclic group of~order $m$. Assume that $1/m \in \Sigma$. The choice of a generator $t \in \rho$ leads to a decomposition 
$$
\Sigma[\rho]\simeq \Sigma[t]/(t^m-1)\simeq \bigoplus_{j\mid m} \Sigma[t]/(\Phi_j(t))\,,
$$
where $\Phi_j(t)$ stands for the cyclotomic polynomial corresponding to the primitive $j^{\mathrm{th}}$ roots of unity. Consider the idempotent $e_{\text{prim}}:=\prod_{\rho'\subsetneq  \rho}(1-e_{\rho'})$, with $e_{\rho'}:=(\sum_{h \in \rho'} h)/|\rho'|$, where the product runs over all minimal non-trivial subgroups $\rho'$ of $\rho$. Note that $e_{\text{prim}}$, being a product of idempotents is, indeed, an idempotent. Moreover, it is mapped to $\delta_{j,m}$ under the projection $\Sigma[\rho]\to \Sigma[t]/(\Phi_j(t))$. The {\em primitive part} of $\Sigma[\rho]$ is defined as follows: 
%
\begin{equation}
\label{eq:primdef}
\Sigma[\rho]_{\text{prim}}:=e_{\text{prim}}\Sigma[\rho]\simeq\Sigma[t]/(\Phi_m(t))\,.
\end{equation}
\begin{remark} \label{rk:char}
If $u:\rho\to \varrho$ is a non-injective group homomorphism, then $u(e_{\text{prim}})=0$. Indeed, since $\mathrm{Ker}(u)$ contains a minimal subgroup $\rho'$, we have $u(1-e_{\rho'})=0$. The idempotent $e_{\text{prim}}$ is maximal with respect to this property.
\end{remark}
The following result follows easily from the above definitions/considerations:
\begin{lemma} \label{lem:perm}
Every automorphism of $\Sigma[\rho]$ that permutes the elements of $\rho$ leaves $e_{\operatorname{prim}}$, and consequently $\Sigma[\rho]_{\operatorname{prim}}$, invariant.
\end{lemma}
\subsubsection{Representation rings of cyclic groups}
\label{sec:prim}
Let $l$ be the field obtained from $k$ by adjoining the $n^{\mathrm{th}}$ roots of unity, and $\Gamma:=\mathrm{Gal}(l/k)$ the associated Galois group. Given a cyclic subgroup $\sigma \in \varphi$, character theory provides an isomorphism $R_l(\sigma)\simeq \bbZ[\sigma^\vee]$, where $\sigma^\vee:=\Hom(\sigma,l^\times)$ stands for the dual cyclic group. Moreover, the $\Gamma$-action on
$\bbZ[\sigma^\vee]$ permutes the elements of $\sigma^\vee$. Therefore, thanks to the above Lemma \ref{lem:perm} (with $\rho=\sigma$ and $\Sigma=\bbZ[1/n]$), the canonical idempotent $e_{\text{prim}}\in  \bbZ[1/n][\sigma^\vee]$ is $\Gamma$-invariant. Using Proposition \ref{prop:galois}, we can then consider $e_{\text{prim}}$ as an element of the representation ring $R(\sigma)_{1/n}\simeq R_l(\sigma)^\Gamma_{1/n}$.
\begin{definition}[Primitive part]\label{def:primitive}
Let $e_\sigma \in R(\sigma)_{1/n}$ be the idempotent corresponding to $e_{\text{prim}} \in \bbZ[1/n][\sigma^\vee]^\Gamma$ under the character isomorphism $R(\sigma)_{1/n} \simeq \bbZ[1/n][\sigma^\vee]^\Gamma$. The {\em primitive part}\footnote{There is no obvious relation between the rings $\widetilde{R}(\sigma)_{1/n}$ and 
$K_0(k[\sigma]_{\text{prim}})$. Indeed, the ring $K_0(k[\sigma]_{\text{prim}})_{1/n}$ is {\em not} a direct summand of
the ring $R(\sigma)_{1/n}$.} $\widetilde{R}(\sigma)_{1/n}$ of $R(\sigma)_{1/n}$ is defined as the direct summand $e_\sigma R(\sigma)_{1/n}$.
\end{definition}
\begin{remark}\label{rk:action}
Thanks to the above Lemma \ref{lem:perm}, the group $\operatorname{Aut}(\sigma)$ of automorphisms of $\sigma$ acts on the primitive part $\widetilde{R}(\sigma)_{1/n}$.
\end{remark}
\begin{lemma} \label{lem:restriction} The idempotent $e_\sigma$ belongs to the kernel of the restriction homomorphism $R(\sigma)_{1/n}\r R(\sigma')_{1/n}$ for every proper subgroup $\sigma'\subsetneq \sigma$.
\end{lemma}
\begin{proof} We may assume without loss of generality that $k$ contains the $n^{\mathrm{th}}$ roots of unity. Therefore, the proof follows from the above Remark \ref{rk:char}.
\end{proof}
\begin{proposition}[Computation] 
\label{prop:concrete}
Assume that $k$ contains the $n^{\mathrm{th}}$ roots of unity. Let $l$ be another field (not necessarily of characteristic $p$) which contains the $n^{\mathrm{th}}$ roots of unity and $1/n\in l$. Choose an isomorphism $\epsilon$ between the $n^{\mathrm{th}}$ roots of unity in $k$ and in $l$ (e.g.\ the identity if $k=l$). Under these assumptions, we have the following commutative diagram of $l$-algebras (compatible with the $\operatorname{Aut}(\sigma)$-action):
\begin{equation}
\label{eq:character}
\xymatrix{
\widetilde{R}(\sigma)_l\ar[r]^-{\simeq}\ar@{^(->}[d]& 
\mathrm{Map}(\mathrm{gen}(\sigma),l)\ar@{^(->}[d]\\ 
R(\sigma)_l\ar[r]_-{\simeq}& \mathrm{Map}(\sigma,l) \,.
}
\end{equation}
Some explanations are in order: $\mathrm{gen}(\sigma)$ stands for the set of generators of $\sigma$; $\mathrm{Map}(\mathrm{gen}(\sigma),l)$ stands for the set of functions from $\mathrm{gen}(\sigma)$ to $l$; the lower arrow in \eqref{eq:character} sends an irreducible $\sigma$-representation $V$ to the composition of its character $\chi_V$ with $\epsilon$; and the right vertical arrow in \eqref{eq:character} identifies a function on $\mathrm{gen}(\sigma)$ to the function on $\sigma$ that is zero on the complement of $\mathrm{gen}(\sigma)$. 
\end{proposition}
\begin{proof}
Note first that we have the following identifications
\begin{equation}
\label{eq:Rk}
R(\sigma)_{{l}}\simeq R(\sigma)_{1/n}\otimes_{\bbZ[1/n]} {{l}}\simeq \bbZ[\sigma^\vee]_{1/n}\otimes_{\bbZ[1/n]} {{l}}\simeq {{l}}[\sigma^\vee]\,.
\end{equation}
By composing them with the following $l$-algebra isomorphism
\begin{eqnarray*}
 {{l}}[\sigma^\vee]\stackrel{\sim}{\too} {{l}}[\sigma]^\vee\simeq  \mathrm{Map}(\sigma,{{l}}) && \chi\mapsto (g\mapsto (\epsilon \circ\chi)(g))\,,
\end{eqnarray*}
we obtain the lower arrow in diagram \eqref{eq:character}. Note also that \eqref{eq:Rk} implies that $\widetilde{R}(\sigma)_{{l}}\simeq{{l}}[\sigma^\vee]_{\operatorname{prim}}$. Thanks to the right-hand side of \eqref{eq:primdef} (with $\Sigma=l$), the dimension of the $l$-vector space $\widetilde{R}(\sigma)_{{l}}$ is equal to $\phi(|\sigma|)$, where $\phi$ stands for Euler's totient function. Now, Lemma \ref{lem:restriction} leads to the following inclusion of $l$-algebras
$$
\widetilde{R}(\sigma)_{{l}}\subset 
\cap_{\sigma'\subsetneq \sigma}\ker(\mathrm{Map}(\sigma,{{l}})
\too \mathrm{Map}(\sigma',{{l}}))\,.
$$
This shows that an element of $\widetilde{R}(\sigma)_{{l}}$ is a function on $\sigma$
which is zero on all proper subgroups $\sigma'$ of $\sigma$. Consequently, we have $\widetilde{R}(\sigma)_{{l}}\subseteq \mathrm{Map}(\mathrm{gen}(\sigma),{{l}})$. Finally, since both sides of this inclusion have the same dimension $\phi(|\sigma|)$, the proof is finished.
\end{proof}

\section{$G_0$-motives over an orbifold}
\label{sec:G0motives}
%
%

Let $X$ be 
a smooth separated $k$-scheme equipped with a $G$-action. In this section, of independent interest, we construct a certain category of $G_0$-motives $\MMot(\cX)$ over the (global) orbifold $\cX:=[X/G]$, as well as a functor $\Psi\colon \MMot(\cX) \to \NMot(k)$. This will be a key technical tool used in the proof of Theorems \ref{thm:main} and \ref{thm:formula-main4}. 

\subsection{Construction} Let
$\DM(\cX)$ be the category whose objects are
the finite morphisms of smooth separated Deligne-Mumford stacks
$\cY \stackrel{f}{\to}\cX$ with the property that the stabilizer
orders in $\cY$ are invertible in $k$. Note that $\DM(\cX)$ is a
2-category. Concretely, a morphism from
$\cY_1 \stackrel{f_1}{\to} \cX$ to $\cY_2 \stackrel{f_2}{\to} \cX$
consists of a pair $(\alpha,\eta)$, where
$\alpha\colon \cY_1 \to \cY_2$ is a 1-morphism and
$\eta\colon f_2 \circ \alpha \Rightarrow f_1$ a
2-isomorphism. Whenever $f_1$, $f_2$, and $\eta$, are clear from the
context, we will omit them from the notation and write simply
$\alpha:\cY_1\r \cY_2$ for a morphism in $\DM(\cX)$.

Let $\HHmo(\cX)$ be the category with the same objects as $\DM(\cX)$.
Given any two objects $\cY$ and $\cY'$ of $\HHmo(\cX)$, let $\Hom_{\HHmo(\cX)}(\cY, \cY')$ be the set of isomorphism classes of the bounded derived category
\begin{equation}\label{eq:coh-support}
\cD^b \mathrm{coh}_{\cY \times_{\cX} \cY'}(\cY \times_{BG} \cY')\subset \perf(\cY \times_{BG} \cY')\,,
\end{equation}
of those coherent $\cO_{\cY \times_{BG} \cY'}$-modules 
that are (topologically) supported on the closed substack $\cY \times_{\cX} \cY'$ of $\cY \times_{BG} \cY'$. 
Note that since $\cY \times_{BG} \cY'$ is smooth, every bounded coherent complex on it
is perfect. Note also that the above definition \eqref{eq:coh-support} depends on the fact that $\cX$ is a (global) quotient stack.
The composition law
$$\Hom_{\HHmo(\cX)}(\cY',\cY'') \times \Hom_{\HHmo(\cX)}(\cY,\cY') \too \Hom_{\HHmo(\cX)}(\cY,\cY'')$$
is induced by the classical (derived) ``pull-back/push-forward'' formula
\begin{equation}\label{eq:assignment}
(\cE',\cE) \mapsto (p_{13})_\ast ((p_{23})^\ast(\cE') \otimes^{\bf L} (p_{12})^\ast(\cE)) \,,
\end{equation}
where $p_{ij}$ stands for the projection from the triple fiber product onto its $ij$-factor. Finally, the identity of an object $\cY$ of $\HHmo(\cX)$ is the (isomorphism class of the) structure sheaf $\cO_\Delta$ of the diagonal $\Delta$ in $\cY \times_{BG} \cY$. 

The {\em additivization} of $\HHmo(\cX)$ is the category $\HHmo_0(\cX)$ with the same objects as $\HHmo(\cX)$ and with abelian groups of morphisms 
$\Hom_{\HHmo_0(\cX)}(\cY,\cY')$ given by the Grothendieck group\footnote{The idea of using the Grothendieck group in the construction of categories of ``motivic nature'' goes back to the work of Manin \cite{Manin} and Gillet-Soul\'e \cite[\S5.2]{GS}.} of the triangulated category \eqref{eq:coh-support}. Thanks to Quillen's d\'evissage theorem \cite[\S5]{Quillen} and to the definition of $G$-theory, we have~isomorphisms
\begin{equation}\label{eq:G-theory}
\Hom_{\HHmo_0(\cX)}(\cY, \cY')\simeq G_0(\cY \times_{\cX} \cY')\,.
\end{equation}
In particular, we have a ring isomorphism
\begin{equation}\label{eq:ring}
\End_{\HHmo_0(\cX)}(\cX)\simeq K_0(\cX)\,.
\end{equation}
Note also that, since the assignment \eqref{eq:assignment} is exact in each one of the variables, the composition law of $\HHmo(\cX)$ extends naturally to $\HHmo_0(\cX)$.
\begin{definition}[$G_0$-motives]
The {\em category of $G_0$-motives $\MMot(\cX)$} is defined by formally adding to $\HHmo_0(\cX)$ all finite direct sums and direct summands.
Let us write $\UU(-)$ for the canonical functor from $\HHmo(\cX)$ to $\MMot(\cX)$.
\end{definition}
Given objects $\cY$ and $\cY'$ of $\HHmo(\cX)$,
let us write $i$ for the finite  morphism
$\cY\times_{BG}\cY'\too \cY \times \cY'$. Under this notation, we have the following exact functor
\begin{eqnarray}
\label{eq:dgfunctor}
& \cD^b \mathrm{coh}_{\cY \times_{\cX} \cY'}(\cY \times_{BG} \cY')\too \rep(\perf_\dg(\cY), \perf_\dg(\cY'))&\cE \mapsto {}_{\Phi_{i_\ast(\cE)}}\mathrm{B}\,,
\end{eqnarray}
where $\Phi_{i_\ast(\cE)}$ stands for the Fourier-Mukai dg-functor
\begin{eqnarray}
\label{eq:fm}
\perf_\dg(\cY)\too \perf_\dg(\cY') \qquad \cG \mapsto (p_2)_\ast((p_1)^\ast(\cG) \otimes^{\bf L} i_\ast(\cE))\,,
\end{eqnarray}
and ${}_{\Phi_{i_\ast(\cE)}}\mathrm{B}$ for the associated bimodule; see \S\ref{sub:dg}. By construction of the
categories $\HHmo(\cX)$ and $\Hmo(k)$, these considerations lead to a functor
\begin{eqnarray*}
\HHmo(\cX)\too \Hmo(k) & \cY \mapsto \perf_\dg(\cY) & \cE \mapsto {}_{\Phi_{i_\ast(\cE)}}\mathrm{B}\,,
\end{eqnarray*}
%
%
which naturally extends to the motivic categories:
\begin{eqnarray*}
\Psi\colon \MMot(\cX) \too \NMot(k) && \UU(\cY) \mapsto \U(\cY)\,.
\end{eqnarray*}
\begin{remark}[Sheaves of algebras]\label{rk:sheaves}
Suppose that $\cX$ is equipped with a \emph{flat} quasi-coherent sheaf of $\cO_{\cX}$-algebras $\cF$. In this case, as we now explain, it is possible to construct a variant $\Psi_\cF$ of the functor $\Psi$. Given objects $\cY_1 \stackrel{f_1}{\to} \cX$ and $\cY_2 \stackrel{f_2}{\to} \cX$ of $\HHmo(\cX)$, we may view \eqref{eq:dgfunctor} as an exact functor
$$
\cD^b \mathrm{coh}_{\cY \times_{\cX} \cY'}(\cY \times_{BG} \cY')\too  \rep(\perf_\dg(\cY_1;f_1^\ast(\cF)), 
\perf_\dg(\cY_2;f_2^\ast(\cF)))\,,
$$
where $\Phi_{i_\ast(\cE)}$ is defined by the above formula \eqref{eq:fm} but considered as a Fourier-Mukai dg-functor from $\perf_\dg(\cY_1;f_1^\ast(\cF))$ to $\perf_\dg(\cY_2;f_2^\ast(\cF))$.
By construction of the
categories $\HHmo(\cX)$ and $\Hmo(k)$, these considerations lead to a functor
\begin{eqnarray*}
\HHmo(\cX)\too \Hmo(k) & (\cY\stackrel{f}{\to} \cX) \mapsto \perf_\dg(\cY; f^\ast(\cF)) & \cE \mapsto {}_{\Phi_{i_\ast(\cE)}} \mathrm{B}\,,
\end{eqnarray*}
which naturally extends to the motivic categories:
\begin{eqnarray*}
\Psi_\cF\colon \MMot(\cX) \too \NMot(k) && \UU(\cY\stackrel{f}{\to} \cX) \mapsto \U(\cY; f^\ast(\cF))\,.
\end{eqnarray*}
%
\end{remark}
\subsection{Properties} In what follows, we establish some (structural) properties of the category of $G_0$-motives. These will be used in the next sections.
\subsubsection{Push-forward and pull-back}\label{rk:pullback}
Let $\alpha\colon \cY_1 \to \cY_2$ be a morphism in $\mathrm{DM}(\cX)$. Its {\em push-forward} $\alpha_\ast\colon \UU(\cY_1) \to \UU(\cY_2)$, resp. {\em pull-back} $\alpha^\ast\colon \UU(\cY_2) \to \UU(\cY_1)$, is defined as the Grothendieck class $[(\id,\alpha)_\ast(\cO_{\cY})]\in G_0(\cY\times_{\cX}\cY')$, resp. $[(\alpha,\id)_\ast(\cO_{\cY})]\in G_0(\cY'\times_{\cX}\cY)$; see \eqref{eq:G-theory}. Note that $\Psi(\alpha_\ast)=\alpha_\ast$ and $\Psi(\alpha^\ast)=\alpha^\ast$. Note also that if $\alpha,\beta \colon \cY_1 \to \cY_2$ are isomorphic 1-morphisms in $\DM(\cX)$, then $\alpha_\ast=\beta_\ast$ and $\alpha^\ast =\beta^\ast$.

\subsubsection{$K_0$-action}\label{rk:K0action}
  Let $\UU(\cY)$ be an object of $\MMot(\cX)$. The push-forward along the diagonal map 
$i_\Delta:\cY\r \cY\times_{BG} \cY$ leads to an exact functor
\begin{equation}\label{eq:diagonal}
i_{\Delta,\ast}:\perf(\cY)\too  \cD^b\mathrm{coh}_{\cY\times_{\cX}\cY}(\cY\times_{BG}\cY)
\end{equation}
that sends the tensor product $-\otimes_{\cY}-$ on the left-hand side to the ``pull-back/push-forward'' formula \eqref{eq:assignment} on the right-hand side. Therefore, by applying $K_0(-)$ to \eqref{eq:diagonal}, we obtain an induced ring morphism $K_0(\cY)\to \End_{\MMot(\cX)}(\UU(\cY))$. In other words, we obtain an action of the Grothendieck ring $K_0(\cY)$ on the $G_0$-motive $\UU(\cY)$.
\begin{lemma} \label{lem:compat} The functor $\Psi$ is compatible with the $K_0(\cY)$-action on $\UU(\cY)$ (defined above) and on $\U(\cY)$ (defined in \S\ref{sec:action}).
\end{lemma}
\begin{proof}
Consider the following commutative diagram:
\begin{equation}\label{eq:diagram}
\xymatrix{
\perf(\cY)\ar@{=}[d]\ar[rrr]^-{i_{\Delta,\ast}}&&& \cD^b\mathrm{coh}_{\cY \times_{\cX} \cY}(\cY \times_{BG} \cY)\ar[d]^{\cE\mapsto  {}_{\Phi_{i_\ast(\cE)}}\mathrm{B}}\\
\perf(\cY)\ar[rrr]_-{M\mapsto{}_{(M\otimes_{\cY}-)}\mathrm{B}}&&& \rep(\perf_{\dg}(\cY),\perf_{\dg}(\cY))\,.
}
\end{equation}
By applying $K_0(-)$ to \eqref{eq:diagram}, we obtain the claimed compatibility.
\end{proof}
\subsubsection{$K_0(\cX)$-linearity}\label{rk:compatibility} 
Let $\cY\stackrel{f}{\to} \cX$ be an object of $\DM(\cX)$. By composing the induced ring homomorphism $f^\ast\colon K_0(\cX) \to K_0(\cY)$ with the $K_0(\cY)$-action on $\UU(\cY)$ described in \S\ref{rk:K0action}, we obtain a $K_0(\cX)$-action on $\UU(\cY)$. A simple verification shows that this $K_0(\cX)$-action is compatible with the morphisms of $\MMot(\cX)$. In other words, $\MMot(\cX)$ is a $K_0(\cX)$-linear category. The map $\cX \to BG$ induces a ring homomorphism $R(G) \to K_0(\cX)$. Therefore, $\MMot(\cX)$ is also a $R(G)$-linear category.
\subsubsection{$G$-action}\label{rk:normalizer} Let $Y\hookrightarrow X$ be a 
smooth closed $k$-subscheme of~$X$ which is preserved by a subgroup $H\subseteq G$, and
$[Y/H]\stackrel{f}{\to} \cX$ the corresponding object of $\DM(\cX)$. Consider also the objects $[gY/gHg^{-1}]\stackrel{f_g}{\to} \cX$ of $\DM(\cX)$ with $g \in G$. Note that we have a commutative $2$-diagram of (global) orbifolds
$$
\xymatrix{
[Y/H] \ar[rr]^-{\alpha_g} \ar[dr]_-f & \ar@{}[d]|-{\overset{\eta_g}{\Leftarrow}}&  [gY/gHg^{-1}] \ar[dl]^-{f_g} \\
& \cX & \,,
}
$$
where the $1$-isomorphism $\alpha_g$ is given by $(y\mapsto gy, h\mapsto ghg^{-1})$ and the evaluation of the $2$-isomorphism $\eta_g\colon f_g \circ \alpha_g \Rightarrow f$ at $y \in Y$ is given by $gy\stackrel{g^{-1}}{\to} y$. Therefore, $\alpha_g$ may be considered as an isomorphism in the category $\DM(\cX)$; in the sequel, we will write $g$ instead of $\alpha_g$ (note that $\alpha_g$ acts as $g$ on $Y$). By functoriality (see \S\ref{rk:pullback}), we obtain inverse isomorphisms $g^\ast$ and $g_\ast$ between $\UU([Y/H])$ and $\UU([gY/gHg^{-1}])$.

Now, assume that $Y$ is stabilized by the normalizer $N(H)$ of $H$. In this case, the above considerations lead to a $N(H)$-action on $[Y/H]$ and, consequently, on the $G_0$-motive $\UU([Y/H])$. Moreover, the induced morphism $f^\ast\colon \UU(\cX)_{1/n} \to \UU([Y/H])_{1/n}$ factors through the direct summand $\UU([Y/H])_{1/n}^{N(H)}$.
\begin{example}\label{ex:action}
Let $\sigma \in \varphi$ be a cyclic subgroup. By choosing $Y=X^\sigma$, resp. $Y=X$, and $H=\sigma$, we obtain an induced $N(\sigma)$-action on $[X^\sigma/\sigma]$, resp. $[X/\sigma]$, and, consequently, on the $G_0$-motive $\UU([X^\sigma/\sigma])$, resp. $\UU([X/\sigma])$.
\end{example}
%
%
\subsubsection{Base-change functoriality}
\label{rk:basefield} Given a field extension $l/k$, base-change $-\times_k l$ leads to a (2-)functor $\DM(\cX)\r \DM(\cX_l)$, where $\cX_l=[X_l/G]$. By construction, this functor extends to $G_0$-motives $(-)\times_k l:\MMot(\cX)\to \MMot(\cX_l), \UU(\cY)\mapsto \UU(\cY_l)$.
\subsubsection{Push-forward functoriality}\label{rk:base}
Let $[Y/H] \to [X/G]$ be a morphism of (global) orbifolds induced by a finite map $Y \to X$ and by a group homomorphism $H \to G$. The associated push-forward functor $\DM([Y/H]) \to \DM([X/G])$ sends $\cZ\r [Y/H]$ to the composition $\cZ\r [Y/H]\r [X/G]$. By construction, this functor extends naturally to $G_0$-motives $\MMot([Y/H])\r \MMot([X/G]), \UU(\cZ) \mapsto \UU(\cZ)$.
\subsubsection{Pull-back functoriality}\label{rk:group}
The morphism of (global) orbifolds $[X/G] \to BG$ leads to a pull-back functor $\MMot(BG) \to \MMot(\cX), \UU(\cY) \mapsto \UU(\cX\times_{BG} \cY)$.
\section{Proofs: decomposition of orbifolds}\label{sec:proof}
In this section we prove Theorem \ref{thm:main} and Corollaries \ref{cor:main1} and \ref{cor:main2}. We start with some preliminaries on representation rings. 
\subsection{Decomposition of representation rings}\label{sub:decomp-representation}
Given a cyclic subgroup $\sigma \in \varphi$, recall from Definition \ref{def:primitive} that the representation ring $R(\sigma)_{1/n}$ comes equipped with a canonical idempotent $e_\sigma$. Since $N(\sigma)$ maps naturally into the group $\mathrm{Aut}(\sigma)$ of automorphisms of $\sigma$, it follows from Remark \ref{rk:action} that $N(\sigma)$ acts on the primitive part $\widetilde{R}(\sigma)_{1/n}$. Moreover, the restriction homomorphism $\mathrm{Res}_\sigma^G\colon R(G)_{1/n} \to R(\sigma)_{1/n}$ factors through the direct summand $R(\sigma)_{1/n}^{N(\sigma)}$; this follows, for example, from the fact that the induced morphism $\UU([\bullet/G])_{1/n} \to \UU([\bullet/\sigma])_{1/n}$ factors through the direct summand $\UU([\bullet/\sigma])_{1/n}^{N(\sigma)}$ (see \S\ref{rk:normalizer}). Under the above notations, we have the following refinement of a result of Vistoli:
\begin{proposition}\label{prop:vistoli}
The following homomorphism of $\bbZ[1/n]$-algebras is invertible:
\begin{equation} \label{eq:vistoli}
R(G)_{1/n}\xrightarrow{\bigoplus_\sigma e_\sigma \circ \mathrm{Res}^G_\sigma}  \bigoplus_{\sigma\in \vphi}\widetilde{R}(\sigma)^{N(\sigma)}_{1/n}\,.
\end{equation}
\end{proposition}
\begin{proof}
In the case where $k$ contains the $n^{\mathrm{th}}$ roots of unity, the isomorphism \eqref{eq:vistoli} was proved in \cite[Prop.~(1.5)]{Vistoli}; see also \cite[Cor.~7.7.10]{Dieck}. Using Proposition \ref{prop:galois}, the general case follows now from Galois descent.
\end{proof}
\begin{remark}
Let $l$ be the field obtained from $k$ by adjoining the $n^{\mathrm{th}}$ roots of unity, and $r$ the degree of the (finite) field extension $l/k$. The above Proposition \ref{prop:vistoli}, with $1/n$ replaced by $1/nr$, was proved by Vistoli in \cite[Thm.~(5.2)]{Vistoli}.
\end{remark} 
\begin{notation}\label{not:idempotents}
Let $\tilde{e}_\sigma \in R(G)_{1/n}$ be the idempotent corresponding to the direct summand $\widetilde{R}(\sigma)_{1/n}^{N(\sigma)}$ under the above isomorphism \eqref{eq:vistoli}.
\end{notation}
\subsection*{Proof of Theorem \ref{thm:main}}
First, recall from \S\ref{rk:compatibility} that the category $\MMot([X/G])$ is $R(G)$-linear. Given a cyclic subgroup $\sigma \in \varphi$, let us write $\gamma_\sigma$ for the associated morphism of (global) orbifolds $[X^\sigma/\sigma] \rightarrow [X/G]$. As explained in Example \ref{ex:action}, the normalizer $N(\sigma)$ acts on the $G_0$-motive $\UU([X^\sigma/\sigma])$. Moreover, the pull-back morphism $\gamma_\sigma^\ast\colon \UU([X/G])_{1/n}\r \UU([X^\sigma/\sigma])_{1/n}$ factors through the direct summand  $\UU([X^\sigma/\sigma])_{1/n}^{N(\sigma)}$. Let $e_\sigma\in R(\sigma)_{1/n}$ be the canonical idempotent introduced in Definition \ref{def:primitive}. As explained in \S\ref{sec:prim}, since $N(\sigma)$ maps naturally into $\mathrm{Aut}(\sigma)$, the idempotent $e_\sigma$ is invariant under the $N(\sigma)$-action on $R(\sigma)_{1/n}$. Hence, we can (and will) identify the $G_0$-motives $e_\sigma(\UU([X^\sigma/\sigma])_{1/n}^{N(\sigma)})$ and $(e_\sigma(\UU[X^\sigma/\sigma])_{1/n})^{N(\sigma)}$. Under the above notations and identifications, we have the following morphism
\begin{equation*}\label{eq:formula-main-new1}
  \bar{\theta}\colon \UU([X/G])_{1/n} \xrightarrow{\bigoplus_\sigma e_\sigma \circ \gamma^\ast_\sigma} \bigoplus_{\sigma \in \vphi} \widetilde{\UU}([X^\sigma/\sigma])^{N(\sigma)}_{1/n}
\end{equation*} 
in the category $\MMot([X/G])_{1/n}$. As explained in \S\ref{rk:pullback}, resp. Lemma \ref{lem:compat}, the functor $\Psi$ is compatible with pull-backs, resp. $K_0$-actions. Therefore, by applying the functor $\Psi$ to $\overline{\theta}$ we obtain the following morphism
\begin{equation*}\label{eq:formula-main-new}
  \theta\colon \U([X/G])_{1/n} \xrightarrow{\bigoplus_\sigma  e_\sigma\circ \gamma^\ast_\sigma} \bigoplus_{\sigma \in \vphi} \widetilde{\U}([X^\sigma/\sigma])^{N(\sigma)}_{1/n}
\end{equation*}
in the category $\NMot(k)_{1/n}$. Thanks to Proposition \ref{prop:key} below, the morphism $\overline{\theta}$ (and consequently $\theta$) is invertible. Since $\theta$ is an isomorphism of monoids, the proof of Theorem \ref{thm:main} follows now from the equivalences of categories~\eqref{eq:equivalence1}-\eqref{eq:equivalence2}. 

\begin{proposition}\label{prop:key}
The above morphism $\overline{\theta}$ is invertible.
\end{proposition}
%
%
\begin{proof} 
Consider the inclusion $\iota$ and projection $\pi$ morphisms
\begin{eqnarray*}
\widetilde{\UU}([X^\sigma/\sigma])^{N(\sigma)}_{1/n} \stackrel{\iota_\sigma}{\too} \UU([X^\sigma/\sigma])_{1/n}  & & \UU([X^\sigma/\sigma])_{1/n} \stackrel{\pi_\sigma}{\too} \widetilde{\UU}([X^\sigma/\sigma])^{N(\sigma)}_{1/n}\,.
\end{eqnarray*}
Under these notations, we have $\overline{\theta}=\bigoplus_{\sigma} \pi_\sigma\circ \gamma^\ast_\sigma$.  Thanks to
Lemmas \ref{lem:vanishing}, \ref{lem:aux1}, and \ref{lem:aux2} below (see Remark \ref{rk:right-inverse}),
the morphism $\overline{\theta}$ admits a right inverse $\bar{\psi}$.
This implies that $e:=\id-\bar{\psi}\circ \bar{\theta}$ is an idempotent
in $\End(\UU([X/G])_{1/n})\simeq K_0([X/G])_{1/n}$.  We need to prove that $e=0$. Let $l$ be the field obtained from $k$ by adjoining the $n^{\mathrm{th}}$ roots of unity. Thanks to Corollary \ref{cor:nilp}, it is sufficient
to prove that $e\times_k l=0$ (consult \S\ref{rk:basefield} for the notation
$(-)\times_k l$). Hence, we may assume without loss of generality that $l=k$ and, consequently, that $k$ contains the $n^{\mathrm{th}}$ roots of unity.

In order to prove that $e=0$, let us consider the additive invariant $K_0(-)_{1/n}$ with values in the category of $\bbZ[1/n]$-modules. The following composition
\begin{equation}
\label{eq:action}
\End(\UU([X/G])_{1/n})\stackrel{\mathrm{(a)}}{\too} \End(\U([X/G])_{1/n})\stackrel{\mathrm{(b)}}{\too} \End(K_0([X/G])_{1/n})\,,
\end{equation}
where (a) is induced by the functor $\Psi$ and (b) by the functor corresponding to $K_0(-)_{1/n}$ under the equivalence of categories \eqref{eq:equivalence1}, sends $[\cE]_{1/n} \in K_0([X/G])_{1/n}$ to left multiplication by $[\cE]_{1/n}$. Since $K_0([X/G])_{1/n}$ is a unital $\bbZ[1/n]$-algebra, this implies that the morphism \eqref{eq:action} is injective. By an abuse of notation, let us still denote by $K_0(-)_{1/n}$ the composition \eqref{eq:action}. Under this notation, $e=0$ if and only if $K_0(e)_{1/n}=0$. By construction, we have $K_0(e)_{1/n}=\id-K_0(\bar{\psi})_{1/n}\circ K_0(\bar{\theta})_{1/n}$. Therefore, it suffices to prove that $K_0(\bar{\theta})_{1/n}$ is an isomorphism.  Modulo the caveat in Remark \ref{rk:ample} below, this was proved in \cite[Thm.~(5.2)]{Vistoli}.
\end{proof}
\begin{remark}[Ample line bundle]\label{rk:ample}
  Vistoli's results \cite{Vistoli} were proved under the assumption that $X$ carries an ample line
  bundle. However, as explained by him in \cite[page~402]{Vistoli}, this
  assumption is only used in
  the proof of \cite[Lem.~1.1]{Vistoli}. By invoking \cite[Thm.~2.3
  and Cor.~2.4]{Azumaya}, we observe that the proof of \cite[Lem.~1.1]{Vistoli} remains valid under the much weaker assumption that $X$ is
  quasi-compact and quasi-separated. Therefore, we can freely use the
results from \cite{Vistoli} is our current setting.
\end{remark}
\begin{remark}[Right-inverse]\label{rk:right-inverse}
Let $\bar{\psi}':=\bigoplus_\sigma \gamma_{\sigma,\ast}\circ \iota_\sigma$. In what follows, we will prove that the composition $\bar{\theta}\circ\bar{\psi}'$ is invertible. This implies that $\bar{\theta}$ admits a right~inverse~$\overline{\psi}$.
\end{remark}
\begin{lemma}\label{lem:vanishing}
For every $\sigma\neq \sigma' \in \vphi$, we have $\pi_{\sigma'}\circ \gamma^\ast_{\sigma'} \circ \gamma_{\sigma,\ast} \circ \iota_\sigma =0$.
\end{lemma}
\begin{proof}
Since the category $\MMot([X/G])$ is $R(G)$-linear, the morphism 
\begin{equation}
\label{eq:source:target}
\widetilde{\UU}([X^\sigma/\sigma])^{N(\sigma)}_{1/n}\xrightarrow{\pi_{\sigma'}\circ \gamma^\ast_{\sigma'} \circ \gamma_{\sigma,\ast} \circ \iota_\sigma} \widetilde{\UU}([X^{\sigma'}/\sigma'])^{N(\sigma')}_{1/n}
\end{equation}
is $R(G)_{1/n}$-linear. The idempotent $\tilde{e}_\sigma\in R(G)_{1/n}$ (see Notation \ref{not:idempotents}) acts as the identity on the source of \eqref{eq:source:target},
whereas the idempotent $\tilde{e}_{\sigma'}$ acts as the identity on the target. Therefore, since $\tilde{e}_\sigma\tilde{e}_{\sigma'}=0$, the homomorphism \eqref{eq:source:target} is zero. 
\end{proof}
We now claim that the following composition is invertible
\begin{equation}
\label{eq:factor2}
\widetilde{\UU}([X^\sigma/\sigma])^{N(\sigma)}_{1/n}\xrightarrow{\pi_{\sigma}\circ \gamma^\ast_\sigma \circ \gamma_{\sigma,\ast} \circ \iota_\sigma} \widetilde{\UU}([X^\sigma/\sigma])^{N(\sigma)}_{1/n}\,.
\end{equation}
Note that this implies that the composition $\bar{\theta}\circ\bar{\psi}'$ is invertible, and hence concludes the proof of Proposition \ref{prop:key}. In order to prove this claim, consider the factorization
$$\gamma_\sigma\colon [X^\sigma/\sigma] \xrightarrow{\alpha_\sigma} [X/\sigma] \xrightarrow{\beta_\sigma} [X/G]\,,$$ as well as the inclusion $\iota'$ and projection $\pi'$ morphisms:
\begin{eqnarray*}
\widetilde{\UU}([X/\sigma])^{N(\sigma)}_{1/n} \stackrel{\iota'_\sigma}{\too} \UU([X/\sigma])_{1/n} & & \UU([X/\sigma])_{1/n} \stackrel{\pi'_\sigma}{\too} \widetilde{\UU}([X/\sigma])^{N(\sigma)}_{1/n}\,.
\end{eqnarray*}
This data leads to the following commutative diagrams
\begin{equation}
\label{eq:bigcomdiag}
\xymatrix{
\widetilde{\UU}([X^\sigma/\sigma])_{1/n}^{N(\sigma)} \ar[r]^-{\iota_\sigma}\ar[d]_{\alpha_{\sigma,\ast}} & \UU([X^\sigma/\sigma])_{1/n}\ar[r]^-{\gamma_{\sigma,\ast}}\ar[d]_{\alpha_{\sigma,\ast}}& 
\UU([X/G])_{1/n} \ar@{=}[d] \\
 \widetilde{\UU}([X/\sigma])_{1/n}^{N(\sigma)}\ar[r]_{\iota'_\sigma}& \UU([X/\sigma])_{1/n}\ar[r]_{\beta_{\sigma,\ast}} & \UU([X/G])_{1/n}
}
\end{equation}
\begin{equation}\label{eq:bigcomdiag2}
\xymatrix{
\UU([X/G])_{1/n} \ar@{=}[d] \ar[r]^-{\gamma_\sigma^\ast}&  \UU([X^\sigma/\sigma])_{1/n}\ar[r]^-{\pi_\sigma}& \widetilde{\UU}([X^\sigma/\sigma])_{1/n}^{N(\sigma)}\\
\UU([X/G])_{1/n} \ar[r]_{\beta^\ast_\sigma} & \UU([X/\sigma])_{1/n}\ar[u]_{\alpha^\ast_\sigma}\ar[r]_{\pi'_\sigma}& \widetilde{\UU}([X/\sigma])_{1/n}^{N(\sigma)}
\ar[u]_{\alpha_\sigma^\ast}
}
\end{equation}
in the category of $G_0$-motives $\MMot([X/G])_{1/n}$; the existence of the left-hand side, resp. right-hand side, square in \eqref{eq:bigcomdiag}, resp. \eqref{eq:bigcomdiag2}, follows from the $R(G)$-linearity of $\MMot([X/G])$ and from the fact that $\alpha_\sigma$ commutes
with the $N(\sigma)$-actions on $[X^\sigma/\sigma]$ and $[X/\sigma]$
(as objects of $\DM([X/G])$) introduced in Example \ref{ex:action}.
\begin{lemma}\label{lem:aux1}
For every $\sigma \in \varphi$, the composition $\pi'_\sigma \circ \beta^\ast_\sigma \circ \beta_{\sigma,\ast} \circ \iota'_\sigma$ is equal to $[N(\sigma):\sigma]\cdot \id$. In particular, it is invertible.
\end{lemma}
\begin{proof}
  Thanks to the pull-back functor $\MMot(BG)\to \MMot([X/G])$ (see
  \S\ref{rk:group}), it suffices to prove Lemma \ref{lem:aux1} in
  the particular case where $X=\bullet$. Let us write $\cH$ for a
  (chosen) set of representatives of the double cosets
  $\sigma \backslash G /\sigma$. The fiber product
  $[\bullet/\sigma]\times_{[\bullet/G]} [\bullet/\sigma]$ decomposes
  into the disjoint union
  $\amalg_{\tau \in \cH} [\bullet/(\sigma \cap \tau \sigma
  \tau^{-1})]$. By unpacking the definitions, we observe that the composition
\begin{equation}\label{eq:composition11}
\UU([\bullet/\sigma])\stackrel{\beta_{\sigma,\ast}}{\too} \UU([\bullet/G])\stackrel{\beta_{\sigma}^\ast}{\too} \UU([\bullet/\sigma])\,,
\end{equation}
is the sum over $\tau\in \cH$ of the compositions
\begin{equation}
\label{eq:disect}
\UU([\bullet/\sigma])\xrightarrow{\mu^\ast} \UU([\bullet/(\sigma \cap \tau \sigma \tau^{-1})])
\xrightarrow{\nu_\ast} \UU([\bullet/\sigma])\,,
\end{equation}
where
$\mu\colon[\bullet/(\sigma \cap \tau \sigma \tau^{-1})]\r [\bullet/\sigma]$ corresponds to the inclusion $\sigma\cap\tau\sigma\tau^{-1}\subseteq\sigma$ and
$\nu\colon[\bullet/(\sigma \cap \tau \sigma \tau^{-1})]\r [\bullet/\sigma]$ to the inclusion
$\sigma \cap \tau \sigma \tau^{-1}\xrightarrow{\tau^{-1}\cdot
  \tau}\sigma$; note that, as explained in  \S\ref{rk:normalizer}, $\nu$ is a morphism in $\DM([\bullet/G])$.

If $\tau\not\in N(\sigma)$, then $e_\sigma\in R(\sigma)_{1/n}$ acts on $\UU([\bullet/(\sigma \cap \tau \sigma \tau^{-1})])_{1/n}$
by it its image in $R(\sigma \cap \tau \sigma \tau^{-1})_{1/n}$. Hence, thanks to Lemma \ref{lem:restriction}, $e_\sigma$ acts as zero. This implies that the pre-composition (of the $\bbZ[1/n]$-linearization) of \eqref{eq:disect} with the inclusion $\widetilde{\UU}([\bullet/\sigma])_{1/n}\hookrightarrow 
\UU([\bullet/\sigma])_{1/n}$ is zero. Therefore, in the evaluation of $\pi'_\sigma \circ \beta^\ast_\sigma \circ \beta_{\sigma,\ast} \circ \iota'_\sigma$, we
only have to consider the terms in \eqref{eq:disect} where $\tau\in N(\sigma)$.

If $\tau\in N(\sigma)$, then $\mu=\id$ and $\nu=\alpha_{\tau^{-1}}$ as defined in \S\ref{rk:normalizer}. Therefore, \eqref{eq:disect} agrees with the action of $\tau^{-1}\in N(\sigma)$ on $\UU([\bullet/\sigma])_{1/n}$. In particular, it induces the identity on the direct summand $\UU([\bullet/\sigma])^{N(\sigma)}_{1/n}$. Now, the number of terms \eqref{eq:disect} where $\tau\in N(\sigma)$ is equal to $[N(\sigma):\sigma]$. This concludes the proof.
\end{proof}
Using Lemma \ref{lem:aux1} and the commutative diagrams \eqref{eq:bigcomdiag}-\eqref{eq:bigcomdiag2}, we observe that in order to prove our claim it suffices to prove that
the composition
\[
\widetilde{\UU}([X^\sigma/\sigma])_{1/n}^{N(\sigma)}\xrightarrow{\alpha_{\sigma,\ast}} \widetilde{\UU}([X/\sigma])_{1/n}^{N(\sigma)} \xrightarrow{\alpha^\ast_{\sigma}} \widetilde{\UU}([X^\sigma/\sigma])_{1/n}^{N(\sigma)}
\]
is invertible. The preceding composition can be re-written as $\pi_\sigma \circ \alpha^\ast_\sigma \circ \alpha_{\sigma,\ast} \circ \iota_\sigma$. Therefore, the proof of our claim follows now automatically from the next result:
\begin{lemma}\label{lem:aux2}
For every $\sigma \in \varphi$, the composition $\pi_\sigma \circ \alpha^\ast_\sigma \circ \alpha_{\sigma,\ast} \circ \iota_\sigma$ is invertible.
\end{lemma}
\begin{proof}
Consider the following factorizations:
\[
\iota_\sigma \colon \widetilde{\UU}([X^\sigma/\sigma])_{1/n}^{N(\sigma)}\xrightarrow{\iota_{\sigma,1}} \widetilde{\UU}([X^\sigma/\sigma])_{1/n} \xrightarrow{\iota_{\sigma,2}}  \UU([X^\sigma/\sigma])_{1/n}
\]
\[
 \pi_\sigma\colon  \UU([X^\sigma/\sigma])_{1/n}   \xrightarrow{\pi_{\sigma,2}}      \widetilde{\UU}([X^\sigma/\sigma])_{1/n}       \xrightarrow{\pi_{\sigma,1}}         \widetilde{\UU}([X^\sigma/\sigma])_{1/n}^{N(\sigma)}\,.
\]
Since $\pi_\sigma \circ \alpha^\ast_\sigma \circ \alpha_{\sigma,\ast} \circ \iota_\sigma=(\pi_{\sigma,2} \circ \alpha^\ast_\sigma \circ \alpha_{\sigma,\ast} \circ \iota_{\sigma,2})^{N(\sigma)}$, 
it suffices to prove that the composition $\pi_{\sigma,2} \circ \alpha^\ast_\sigma \circ \alpha_{\sigma,\ast} \circ \iota_{\sigma,2}$ is invertible. Moreover, thanks to the push-forward functor $\MMot([X/\sigma]) \to \MMot([X/G])$ (see \S \ref{rk:base}), it is enough to consider the particular case where $G=\sigma$. In this case, we have $\pi_{\sigma,2}=\pi_\sigma$ and $\iota_{\sigma,2}=\iota_\sigma$. Moreover, we have a ring isomorphism (see \S\ref{sec:G0motives})
\begin{equation}\label{eq:ring-iso}
\End_{\MMot([X/\sigma])}(\UU([X^\sigma/\sigma]) \simeq G_0([X^\sigma/\sigma]\times_{[X/\sigma]}[X^\sigma/\sigma])\simeq K_0([X^\sigma/\sigma])\,.
\end{equation}
Furthermore, under the ring isomorphism \eqref{eq:ring-iso}, the composition 
\begin{equation}\label{eq:composition}
\UU([X^\sigma/\sigma])\stackrel{\alpha_{\sigma,\ast}}{\too} \UU([X/\sigma])\stackrel{\alpha^\ast_\sigma}{\too} \UU([X^\sigma/\sigma])\,,
\end{equation}
corresponds to  the following Grothendieck class
\begin{equation*}
[\cO_{[X^\sigma/\sigma]} \otimes^{\bf L}_{X/\sigma} \cO_{[X^\sigma/\sigma]}]=\sum_i (-1)^i [H^i (\cO_{X^\sigma}\otimes^{\bf L}_{X/\sigma} \cO_{X^\sigma})]= \sum_i (-1)^i [\wedge^i (I/I^2)]\,,
\end{equation*}
where $I$ stands for the sheaf of ideals associated to the closed immersion $X^\sigma \hookrightarrow~X$.
By composing and pre-composing (the $\bbZ[1/n]$-linearization of) \eqref{eq:composition} with $\pi_\sigma$ and $\iota_\sigma$, 
respectively, we obtain the image $\xi$ of the above Grothendieck class in the direct summand $e_\sigma K_0([X^\sigma/\sigma])_{1/n}$. In order to prove that $\xi$ is invertible, consider the field $l$ obtained from $k$ by adjoining the $n^{\mathrm{th}}$ roots of unity, and the homomorphism
\begin{equation}\label{eq:induced1}
e_\sigma K_0([X^\sigma/\sigma])_{1/n} \too e_\sigma K_0([X^\sigma_l/\sigma])_{1/n}\,.
\end{equation}
Thanks to Proposition \ref{prop:iso} below (with $X=X_l$), the right-hand side of \eqref{eq:induced1} is isomorphic to $K_0(X^\sigma_l)_{1/n} \otimes_{\bbZ[1/n]} \widetilde{R}_l(\sigma)_{1/n}$. Let us write $\{X^\sigma_{l,i}\}_{i \in I}$ for the connected components of $X^\sigma_l$. Under these notations, we have also the rank map:
\begin{equation}\label{eq:induced2} 
e_\sigma K_0([X^\sigma_l/\sigma])_{1/n} \simeq \bigoplus_{i\in I} K_0(X^\sigma_{l,i})_{1/n} \otimes_{\bbZ[1/n]} \widetilde{R}_l(\sigma)_{1/n}\xrightarrow{\text{rank}}\bigoplus_{i\in I} \widetilde{R}_l(\sigma)_{1/n}\,.
\end{equation}
Let us denote by $\overline{\xi}$ the image of $\xi$ under the composition \eqref{eq:induced2}$\circ$\eqref{eq:induced1}. As proved in \cite[Lem.~(1.8)]{Vistoli} (with the caveat of Remark \ref{rk:ample}), the image of $\xi$ under the homomorphism \eqref{eq:induced1} is invertible. Consequently, $\overline{\xi}$ is also invertible. Thanks to Corollary \ref{cor:nilp} (with $X=X^\sigma$ and $G=\sigma$), resp. \cite[Thm.\ 2.3]{Azumaya}, the elements in the kernel of the homomorphism \eqref{eq:induced1}, resp. \eqref{eq:induced2}, are nilpotent. Therefore, in order to prove that $\xi$ is invertible, it suffices to show that the inverse $\zeta$ of $\overline{\xi}$ belongs to the image of the composition \eqref{eq:induced2}$\circ$\eqref{eq:induced1}. Since $\bigoplus_{i\in I} \widetilde{R}_l(\sigma)_{1/n}$ is a Noetherian $\bbZ[1/n]$-module, there exists an integer $N\geq 1$ such that $\zeta^N+b_1 \zeta^{N-1}+\cdots+b_N=0$ with $b_i \in \bbZ[1/n]$. By multiplying this equality with $\bar{\xi}^N$, we hence conclude that
$$1+b_1\bar{\xi}+\cdots +b_N\bar{\xi}^N=1+(b_1+\cdots +b_N\bar{\xi}^{N-1})\bar{\xi}=0\,.$$
This shows that the inverse $\zeta=-(b_1+\cdots +b_N\bar{\xi}^{N-1})$ of $\overline{\xi}$ belongs, indeed, to the image of the composition \eqref{eq:induced2}$\circ$\eqref{eq:induced1}.
\end{proof}
\subsection*{Proof of Corollary \ref{cor:main1}}
We start with some computations:
%
%
\begin{proposition}\label{prop:iso} Let $\sigma\in \varphi$ be a cyclic subgroup. If $k$ contains the $n^{\mathrm{th}}$ roots of unity, then we have the following isomorphism of (commutative) monoids:
\begin{eqnarray}\label{eq:action6}
\widetilde{\U}([X^\sigma/\sigma])_{1/n}\simeq \U(X^\sigma)_{1/n} \otimes_{\bbZ[1/n]} \widetilde{R}(\sigma)_{1/n}\,.
\end{eqnarray}
%
\end{proposition}
\begin{proof}
Recall from Remark \ref{rk:Homs} that we have the following isomorphism
$$ \Hom_{\NMot(k)}(\U(k),\U(B\sigma))\simeq K_0(B\sigma)=R(\sigma)\,.$$
This leads naturally to a $R(\sigma)$-action on $\U(B\sigma)$, \ie to a morphism of monoids
\begin{equation}
\label{eq:action4}
\U(k)\otimes_\bbZ R(\sigma) \too \U(B\sigma)\,.
\end{equation}
Note that by applying $\Hom_{\NMot(k)}(\U(k),-)$ to \eqref{eq:action4}, we obtain an isomorphism. Therefore, thanks to the enriched Yoneda lemma, in order to show that \eqref{eq:action4} is an isomorphism, it suffices to show that $\U(B\sigma)$ is isomorphic to a (finite) direct sum of copies of $\U(k)$. This is, indeed, the fact because $\U(B\sigma)\simeq \U(k[\sigma])$ and $k[\sigma]$ is isomorphic to a (finite) direct sum of copies of $k$ (this uses the assumption that $k$ contains the $n^{\mathrm{th}}$ roots of unity). Now, note that the cyclic subgroup $\sigma \in \varphi$ acts trivially on $X^\sigma$. By combining Corollary \ref{cor:monoids} (with $X=X^\sigma$ and $G=\sigma$) with \eqref{eq:action4}, we then obtain an isomorphism of (commutative) monoids:
\begin{equation}\label{eq:action5}
\U([X^\sigma/\sigma])\simeq\U(X^\sigma) \otimes \U(k) \otimes_{\bbZ[1/n]} R(\sigma) \simeq \U(X^\sigma) \otimes_{\bbZ[1/n]}R(\sigma)\,.
\end{equation}
Under the isomorphism \eqref{eq:action5}, the canonical $R(\sigma)$-action on $\U([X^\sigma/\sigma])$ (described in \S\ref{sec:action}) corresponds to the tautological $R(\sigma)$-action on $R(\sigma)$. Therefore, the above isomorphism of (commutative) monoids \eqref{eq:action6} is obtained from \eqref{eq:action5} by applying $(-)_{1/n}$ and then by taking the direct summands corresponding to $e_\sigma \in \widetilde{R}(\sigma)_{1/n}$.
\end{proof}
Item (i) of Corollary \ref{cor:main1} follows automatically from the combination of Proposition \ref{prop:iso} with the equivalences of categories \eqref{eq:equivalence1}-\eqref{eq:equivalence2}. Let us now prove item (ii). Thanks to the assumption on the additive invariant $E$ and to Proposition \ref{prop:concrete}, we have the following isomorphisms
$$
E(X^\sigma)\otimes_{\bbZ[1/n]} \widetilde{R}(\sigma)_{1/n}\simeq E(X^\sigma)\otimes_l \widetilde{R}(\sigma)_l\simeq E(X^\sigma) \otimes_l \Map(\mathrm{gen}(\sigma), l)\,.
$$
Therefore, the right-hand side of \eqref{eq:formula-main2} reduces to
\begin{eqnarray}
\bigoplus_{\sigma \in \varphi\!/\!\sim} (E(X^\sigma) \otimes_l \Map(\mathrm{gen}(\sigma),l))^{N(\sigma)} & \simeq & (\bigoplus_{\sigma \in \varphi} E(X^\sigma) \otimes_l \Map(\mathrm{gen}(\sigma),l))^G \nonumber \\
&\simeq & (\bigoplus_{g \in G} E(X^g))^G\,, \label{eq:star-1}
\end{eqnarray}
where \eqref{eq:star-1} follows from the fact that $G=\amalg_{\sigma\in \varphi} \mathrm{gen}(\sigma)$ and $X^{\langle g\rangle}=X^g$. This concludes the proof of item (ii) and, consequently, of Corollary \ref{cor:main1}.
%
\subsection*{Proof of Corollary \ref{cor:main2}}
Note that $\sigma$ acts trivially on $X^\sigma$. Therefore, thanks to Corollary \ref{cor:monoids} (with $X=X^\sigma$ and $G=\sigma$), we have an isomorphism of (commutative) monoids between $\U([X^\sigma/\sigma])$ and $\U(X^\sigma) \otimes \U(B\sigma)$. Under this isomorphism, the canonical $R(\sigma)$-action on $\U([X^\sigma/\sigma])$ (described in \S\ref{sec:action}) corresponds to the tautological $R(\sigma)$-action on $\U(B\sigma)$. Consequently, by applying $(-)_{1/n}$ and then taking the direct summands corresponding to $e_\sigma \in \widetilde{R}(\sigma)_{1/n}$, we obtain an induced isomorphism of (commutative) monoids between $\widetilde{\U}([X^\sigma/\sigma])_{1/n}$ and $\U(X)_{1/n} \otimes \widetilde{\U}(B\sigma)_{1/n}$. Finally, since the additive invariant $E$ is monoidal, the equivalence of categories \eqref{eq:equivalence2} leads to an isomorphism of (commutative) monoids between $\widetilde{E}([X^\sigma/\sigma])$ and $E(X^\sigma) \otimes \widetilde{E}(B\sigma)$. This finishes the proof.
\section{Proofs: smooth quotients}
In this section we prove Theorem \ref{thm:main2}. In order to simplify the exposition, let us write $Y$ for the coarse moduli space $X/\!\!/G$; see \cite{KM1}. We start with some reductions. Firstly, we may (and will) assume that $X$ is connected. 
Secondly, we can assume without loss of generality that $G$ acts generically free; otherwise, simply replace $G$ by $G/N$ where $N$ stands for the generic stabilizer of $X$. Following \S\ref{rk:pullback}, consider the induced pull-back $\pi^\ast\colon \UU(Y)_{1/n} \to \UU(X)_{1/n}$ and push-forward $\pi_\ast\colon \UU(X)_{1/n} \to \UU(Y)_{1/n}$ morphisms in the category\footnote{Note that since $Y$ may be an
    algebraic space, we are not necessarily in the setting of
    \S\ref{sec:G0motives} (with trivial $G$). However, the generalization of \S\ref{sec:G0motives} to
    algebraic spaces is purely formal.} $\MMot(Y)_{1/n}$. The proof will consist on showing that $\pi^\ast$ induces an isomorphism $\UU(Y)_{1/n} \simeq \UU(X)_{1/n}^G$. Similarly to the proof of Theorem \ref{thm:main}, by applying the functor $\Psi$ to the latter isomorphism, we then conclude that $\U(Y)_{1/n} \simeq \U(X)_{1/n}^G$.
 
By construction of the category $\MMot(Y)$, the composition $\pi_\ast \pi^\ast$ is equal to the
class $\xi:=[\pi_\ast(\Oscr_X)]_{1/n} \in K_0(Y)_{1/n} \simeq\End(\UU(Y)_{1/n})$. Therefore, since $\pi_\ast(\Oscr_X)$ has rank $n$, it follows from\footnote{The proof of \cite[Cor.~2.4]{Azumaya} holds similarly for algebraic spaces: simply replace the Zariski topology by the Nisnevich topology and \cite[Prop.~3.3.1]{BO} by \cite[Tag 08GL, Lem.~62.8.3]{Stacks}.} \cite[Cor.~2.4]{Azumaya} that $\xi$ is invertible in $\End(\UU(Y)_{1/n})$. 

Consider the endomorphisms $e:=(1/\xi)(\pi^\ast\pi_\ast)$ (recall from \S\ref{rk:compatibility} that the category $\MMot(Y)_{1/n}$ is $K_0(Y)_{1/n}$-linear) and $e':=(1/n)(\sum_{g\in G} g_\ast)$ of the $G_0$-motive $\UU(X)_{1/n}$.
Both $e$ and $e'$ are idempotents. Moreover, since $\pi$ is $G$-equivariant, we have $e'e=ee'=e$. We claim that $e=e'$. Note that this claim implies that $\pi^\ast$ and $(1/\xi) \pi_\ast$
define inverse isomorphisms between $\UU(Y)_{1/n}$ and $\UU(X)_{1/n}^G$.

In order to prove the preceding claim, since $e'-e$ is also an idempotent, it suffices to show that $e-e'$ is nilpotent.
Let ${{\Gamma}}:=\amalg_{g\in G} \Gamma_g$, where $\Gamma_g\subset X\times_k
  X$ stands for the graph of $g$, and $\gamma\colon {{\Gamma}}\r X\times_Y X$ the map whose restriction to $\Gamma_g$ is given by the inclusion $\Gamma_g\subset X\times X$. Under these notations, the endomorphisms $e$ and $e'$ are represented in $\End(\U(X)_{1/n}) \simeq G_0(X\times_Y X)_{1/n}$ by the Grothendieck classes $(1/\xi)[\Oscr_{X\times_Y X}]$ and $(1/n)[\gamma_\ast(\Oscr_{{\Gamma}})]$, respectively. Since $G$ acts generically free, the map $\gamma$ is birational. Moreover, using the fact that $Y$ is the orbit space for $X$, we have 
$(X\times_Y X)_{\text{red}}=\cup_{g\in G} \Gamma_g\subset X\times
X$ (in this case $X\times_Y X$ is already reduced!). Since generically the class $\xi$ corresponds to multiplication by $n$, the difference $e-e'$ is represented by a class in $F^1 G_0(X\times_Y X)_{1/n}$, where $\{F^j G_0(X\times_Y X)_{1/n}\}_{j\geq 0}$ stands for the (decreasing) codimension filtration of $G_0(X\times_Y X)_{1/n}$. This implies that in order to prove that $e-e'$ is nilpotent, it suffices to show that the composition law of $\End(\UU(X)_{1/n})$ is compatible with the codimension filtration. Clearly, we have  $F^jG_0(X\times_Y X)_{1/n}=\sum_{g\in G} F^jG_0(\Gamma_g)_{1/n}$. Therefore, the following homomorphism 
\begin{equation}
\label{eq:norm}
\gamma_\ast: K_0({{\Gamma}})_{1/n}\too G_0(X\times_Y X)_{1/n}\simeq \End_{\MMot(Y)_{1/n}}(\UU(X)_{1/n})
\end{equation}
induces a surjection $F^jK_0({{\Gamma}})_{1/n}\twoheadrightarrow F^jG_0(X\times_Y X)_{1/n}$ for every $j \geq 0$. Note that we may view
${{\Gamma}}\xymatrix@1{\ar@<0.5ex>[r]^{\pr_1}\ar@<-0.5ex>[r]_{\pr_2}&} X$ as a
groupoid in the category of $k$-schemes. Under this viewpoint, the homomorphism \eqref{eq:norm} sends the convolution product on
$K_0({{\Gamma}})_{1/n}$ to the composition law of $\End(\UU(X)_{1/n})$. Therefore, the proof of Theorem \ref{thm:main2} follows now automatically from Lemma \ref{lem:convolution} below.
\begin{lemma}\label{lem:convolution} Let $X$ be a smooth $k$-scheme and
  ${{\Gamma}}\xymatrix@1{\ar@<0.5ex>[r]^{s}\ar@<-0.5ex>[r]_t&} X$ a groupoid
  in the~category of $k$-schemes with $s, t$ finite \'etale maps. Under these assumptions, the convolution product
  $\cup\colon K_0({{\Gamma}})\times K_0({{\Gamma}})\r K_0({{\Gamma}})$ preserves~the~codimension~filtration.
\end{lemma}
\begin{proof}
The convolution product $\cup$ may be written as the following composition
\[
K_0({{\Gamma}})\otimes K_0({{\Gamma}})\xrightarrow{-\boxtimes-} K_0({{\Gamma}}\times_k {{\Gamma}})\xrightarrow{i^\ast} K_0({{\Gamma}}\times_{s,X,t} {{\Gamma}})
\xrightarrow{\mu_\ast} K_0({{\Gamma}})\,,
\]
where $\mu$ stands for the multiplication map ${{\Gamma}}\times_{s,X,t} {{\Gamma}}\r {{\Gamma}}$ and $i$ for the inclusion ${{\Gamma}}\times_{s,X,t} {{\Gamma}}\r {{\Gamma}}\times_k \Gamma$. Thanks to \cite[Lem.~82]{Gillet}, resp. \cite[Thm.~83]{Gillet}, $-\boxtimes-$, resp. $i^\ast$, preserves the codimension filtration. Since $\mu_\ast$ is a finite map, it also preserves the codimension filtration. Hence, the proof is finished.
\end{proof}
\section{Proofs: equivariant Azumaya algebras}\label{sec:Proofs3}
In this section we prove Theorem \ref{thm:formula-main4} and Corollary \ref{cor:az1}. Recall from \S\ref{sec:intro} that $\cF$ is a flat quasi-coherent sheaf of algebras over $[X/G]$, that $\cF_\sigma$ stands for the pull-back of $\cF$ along the morphism $[X^\sigma/\sigma] \to [X/G]$, that $\cZ_\sigma$ stands for the center of the $\sigma$-graded sheaf of $\cO_{X^\sigma}$-algebras $\cF_\sigma \#\sigma$, and that $Y_\sigma := \uSpec(\cZ_\sigma)$. We start with a (geometric) result concerning sheaves of Azumaya algebras:
\begin{proposition}\label{prop:Azumaya} 
Assume that $\cF$ is a sheaf of Azumaya algebra over $[X/G]$. Under this assumption, the following holds:
\begin{itemize}
\item[(i)] The $\sigma$-graded sheaf of $\cO_{X^\sigma}$-algebra $\cZ_\sigma$ is {\em strongly graded} in the sense of \cite{NVO}. Concretely, $(\cZ_\sigma)_g$ is a line
bundle on $X^\sigma$ for every $g\in \sigma$ and the multiplication on $\cZ_\sigma$ induces an isomorphism $(\cZ_\sigma)_g\otimes_{X^\sigma} (\cZ_\sigma)_h\simeq (\cZ_{\sigma})_{gh}$.
\item[(ii)] The multiplication map induces an isomorphism $\cF_\sigma\otimes_{X^\sigma} \cZ_\sigma\simeq \cF_\sigma\# \sigma$ of $\sigma$-graded $\cO_{X^\sigma}$-algebras. Hence, $\cF_\sigma\#\sigma$ is a sheaf of Azumaya algebras~over~$\cZ_\sigma$.
\item[(iii)] The sheaf $\cZ_\sigma$ is equipped with a unique flat connection which is compatible with its algebra structure and which extends the tautological connection on $\Oscr_{X^\sigma}$.
This connection
is moreover compatible with the $\sigma$-grading.
\end{itemize}
\end{proposition}
\begin{proof} 
Since $\cF_\sigma$ is a sheaf of Azumaya algebras over $X^\sigma$, the functors
\begin{eqnarray*} 
\coh(X^\sigma) \stackrel{\cF_\sigma\otimes_{X^\sigma}-}{\too} \coh(X^\sigma; \cF^\op_{\sigma}\otimes_{X^\sigma} \cF_\sigma) && \coh(X^\sigma; \cF^\op_{\sigma}\otimes_{X^\sigma} \cF_\sigma)  \stackrel{(-)^{\cF_\sigma}}{\too} \coh(X^\sigma)
%
%
\end{eqnarray*}
are inverse monoidal equivalences of categories. Let $\cZ_\sigma':=(\cF_\sigma\# \sigma)^{\cF_\sigma}$ be the centralizer of $\cF_\sigma$ in $\cF_\sigma\# \sigma$. Thanks to the preceding equivalences, the multiplication map $\cF_\sigma\otimes_{X^\sigma} \cZ'_\sigma\r \cF_\sigma\# \sigma$ is an isomorphism of $\cO_{X^\sigma}$-algebras. Moreover, the $\sigma$-grading on $\cF_\sigma\#\sigma$ induces a $\sigma$-grading on $\cZ'_\sigma$ with $(\cZ'_\sigma)_g=(\cF_\sigma g)^{\cF_\sigma}$. Since the $\cF_\sigma$-bimodule
$\cF_\sigma g$ is invertible and $\cF_\sigma g\otimes_{\cF_\sigma} \cF_\sigma h\simeq \cF_\sigma gh$, we then conclude that $\cZ'_\sigma$ is strongly graded. We now claim that $\cZ'_\sigma$ is commutative. Note that thanks to the isomorphism $\cF_\sigma\otimes_{X^\sigma} \cZ'_\sigma\simeq \cF_\sigma\# \sigma$, this claim implies that $\cZ'_\sigma=\cZ_\sigma$, and consequently proves items (i)-(ii). Let $g\in \operatorname{gen}(\sigma)$.  Since $\sigma$ is a cyclic group and $(\cZ'_{\sigma})_{g}$ is locally principal, the
sections of $(\cZ'_{\sigma})_{g}$ commute between themselves. Consequently, we have $(\cZ'_{\sigma})_{g^a}=((\cZ'_{\sigma})_g)^a$. This shows that $\cZ'_{\sigma}$ is commutative.

Let us now prove item (iii). Since $\cZ_\sigma$ is \'etale over $\cO_{X^\sigma}$ (this follows from the fact that $\cZ_\sigma$ is strongly graded), every local
derivation of $\Oscr_{X^\sigma}$ extends uniquely to a local derivation of $\cZ_\sigma$. This leads to 
the unique flat connection on $\cZ_\sigma$ extending the one on $\cO_{X^\sigma}$. Via a local computation, it can be checked that this connection respects the grading; alternatively, first base-change to a field which contains the $n^{\mathrm{th}}$ roots of unity and then use the fact that
the grading corresponds to a $\sigma^\vee$-action which is necessarily compatible with the unique connection.
\end{proof}
\subsection*{Proof of Theorem \ref{thm:formula-main4}}
By applying the functor $\Psi_\cF$ of Remark \ref{rk:sheaves} to the isomorphism $\overline{\theta}$ used in the proof of Theorem \ref{thm:main}, we obtain an induced isomorphism
\begin{equation}\label{eq:iso-sheaf}
\U([X/G];\cF)_{1/n} \simeq \bigoplus_{\sigma \in \vphi} \widetilde{\U}([X^\sigma/\sigma];\cF_\sigma)_{1/n}^{N(\sigma)}
\end{equation}
in $\NMot(k)_{1/n}$; we are implicitly using the obvious analogues for $\Psi_{\cF}$ of
Lemma \ref{lem:compat} and \S\ref{rk:pullback}. The proof follows now from the equivalence of categories \eqref{eq:equivalence1}.
\subsection*{Proof of Corollary \ref{cor:az1}(i)}
The structure map $Y_\sigma \to X^\sigma$ is finite, \'etale, and $\cZ_{\sigma}^{\sigma^\vee}= \cO_{X^\sigma}$. Hence, it is a $\sigma^\vee$-Galois cover. The fact that $\cL_g:=(\cZ_{\langle g\rangle})_g$ is a line bundle (equipped with a flat connection) on $X^\sigma$ follows from~Proposition~\ref{prop:Azumaya}. 
\subsection*{Proof of Corollary \ref{cor:az1}(ii)}
Note first that $\U([X^\sigma/\sigma];\cF_\sigma)$ is canonically isomorphic to $\U(X^\sigma; \cF_\sigma\# \sigma)$. Recall from \S\ref{sec:intro} that $\U(X^\sigma; \cF_\sigma\# \sigma)_{1/n}$ carries a $\bbZ[\sigma^\vee]_{1/n}$-action, where $\sigma^\vee$ stands for the dual cyclic group. Moreover, in the case where $k$ contains the $n^{\mathrm{th}}$ roots of unity, we have a character isomorphism $R(\sigma) \simeq \bbZ[\sigma^\vee]$. Therefore, in this case, we can consider the direct summand $e_\sigma \U(X^\sigma; \cF_\sigma\# \sigma)_{1/n}$ associated to the canonical idempotent $e_\sigma \in R(\sigma)_{1/n}$.
\begin{proposition} If $k$ contains the $n^{\mathrm{th}}$ roots of unity, then \eqref{eq:iso-sheaf} reduces to 
\begin{equation}
\label{eq:algebras31}
\U([X/G]; \cF)_{1/n}\simeq\bigoplus_{\sigma \in \vphi} (e_\sigma \U(X^\sigma; \cF_\sigma\#\sigma)_{1/n})^{N(\sigma)}\,.
\end{equation}
\end{proposition}
\begin{proof}
Given a character $\chi\in \sigma^\vee$, let us write $V_\chi$ for the associated
  $1$-dimensional $\sigma$-representation and $\tau_\chi$ for the
  automorphism of $\cF_\sigma\# \sigma$ corresponding to the
  $\sigma$-grading. Note that $\tau_\chi$ acts trivially on $\cF_\sigma$ and that $\tau_\chi(g)=\chi(g)g$ for every $g\in \sigma$. A simple verification shows that we have following commutative diagram:
\begin{equation}\label{eq:square1}
\xymatrix{
\perf_{\dg}([X^\sigma/\sigma];\cF_\sigma)\ar[d]_{\simeq}\ar[rr]^{V_{\chi}\otimes_k-} &&\perf_{\dg}([X^\sigma/\sigma];\cF_\sigma)\ar[d]^{\simeq}\\
\perf_{\dg}(X^\sigma; \cF_\sigma\# \sigma)\ar[rr]_{\tau_{\chi}^\ast} &&\perf_{\dg}(X^\sigma; \cF_\sigma\# \sigma)\,.
}
\end{equation}
Consequently, by applying the universal additive invariant to \eqref{eq:square1}, we conclude that 
%
%
the $R(\sigma)$-action on $\U([X^\sigma/\sigma]; \cF_\sigma)$ corresponds to the $\bbZ[\sigma^\vee]$-action on $\U(X^\sigma; \cF_\sigma\# \sigma)$ associated to the $\sigma$-grading. The proof follows now from isomorphism \eqref{eq:iso-sheaf}.
\end{proof}
Recall that $\cF$ is a sheaf of Azumaya algebras over $[X/G]$. Thanks to Proposition \ref{prop:Azumaya}(ii), $\cF\# \sigma$ becomes a sheaf of Azumaya algebras over $\cZ_\sigma$. Consequently, using \cite[Thm.~2.1]{Azumaya}, we conclude that 
\begin{equation}\label{eq:iso-Azumaya-1}
\U(X^\sigma; \cF_\sigma \# \sigma)_{1/r} \simeq \U(X^\sigma; \cZ_\sigma)_{1/r} = \U(Y_\sigma)_{1/r}\,,
\end{equation}
where $r$ stands for the product of the ranks of $\cF$ (at each one of the connected components of $X$). Note that the isomorphism \eqref{eq:iso-Azumaya-1} is preserved by the $\sigma^\vee$-action because it is induced by the (inverse) of the following Morita equivalence:
$$ -\otimes_{\cZ_\sigma} (\cF_\sigma \# \sigma) \colon \perf_\dg(X^\sigma; \cZ_\sigma) \too \perf_\dg(X^\sigma; \cF_\sigma \# \sigma)\,.$$
Therefore, the proof of Corollary \ref{cor:az1}(ii) follows now from the combination of the above isomorphisms \eqref{eq:algebras31} and \eqref{eq:iso-Azumaya-1} with the equivalence of categories \eqref{eq:equivalence1}.
\subsection*{Proof of Corollary \ref{cor:az1}(iii)}
Let $l$ be a field which contains the $n^{\mathrm{th}}$ roots of unity and $1/nr\in l$. Recall from \S\ref{sec:intro} that the noncommutative motive $\U(X^\sigma; \cF_\sigma \# \sigma)_l$ may be considered as a $\sigma$-graded object in $\NMot(k)_l$ as soon as we choose an isomorphism $\epsilon$ between the $n^{\mathrm{th}}$ roots of unity in $k$ and in $l$ (\eg\ the identity if $k=l$). In particular, $\U(X^\sigma; \cF_\sigma \# \sigma)_{l, g}$ stands for the degree $g$ part of $\U(X^\sigma; \cF_\sigma \# \sigma)_l$.
\begin{proposition}
If $k$ contains the $n^{\mathrm{th}}$ roots of unity and $l$ is a field which contains the $n^{\mathrm{th}}$ roots of unity and $1/n\in l$, then \eqref{eq:algebras31} reduces to an isomorphism  
\begin{equation}
\label{eq:algebras41}
\U([X/G];\cF)_l\simeq (\bigoplus_{g\in G} \U(X^g; \cF_g\#\langle g\rangle)_{l,g})^G\,.
\end{equation}
\end{proposition}
\begin{proof}
Recall first, from the proof of Proposition \ref{prop:concrete}, that we have the identification
\begin{eqnarray}\label{eq:identification-again}
 {{l}}[\sigma^\vee]\stackrel{\sim}{\too} {{l}}[\sigma]^\vee\simeq \mathrm{Map}(\sigma,{{l}}) && \chi\mapsto (g\mapsto (\epsilon \circ\chi)(g))\,.
\end{eqnarray}
Let $\sigma\in \varphi$ be a (fixed) cyclic subgroup. Note that for every $g\in \sigma$ we have
\[
\U(X^g; \cF_g\#\langle g\rangle\bigr)_{l,g}=e_g\U(X^g;\cF_g\#\langle g\rangle)_{l}\,,
\]
where $e_g \in l[\sigma^\vee]$ stands for the idempotent $(1/|\sigma|)(\sum_{\chi\in \sigma^\vee} (\epsilon\circ \chi)(g)^{-1}\chi)$. Under the above identification \eqref{eq:identification-again}, the idempotent $e_\sigma$ corresponds to the characteristic function of $\mathrm{gen}(\sigma)$ (since this is the unit element of $\Map(\mathrm{gen}(\sigma),l) \subset \Map(\sigma, l)$) and the idempotent $e_g$ to the characteristic function of $g$. Therefore, the equality $e_\sigma= \sum_{g\in\operatorname{gen}(\sigma)} e_g$ holds. Consequently, the right-hand side of \eqref{eq:algebras31} reduces to
$$
(\bigoplus_{\sigma \in \varphi} e_\sigma \U(X^\sigma; \cF_\sigma\#\sigma)_{l})^{G}  \simeq  (\bigoplus_{\sigma, g} e_g \U(X^\sigma; \cF_\sigma\#\sigma)_{l})^{G} \simeq  (\bigoplus_{g\in G}  \U(X^g; \cF_g\#\langle g\rangle)_{l,g})^{G}\,,
$$
where the middle direct sum runs over the pairs $\sigma \in \varphi$ and $g\in \mathrm{gen}(\sigma)$. 
\end{proof}
%
%
%
%
Recall that $\cF$ is a sheaf of Azumaya algebras over $[X/G]$. Similarly to the proof of Corollary \ref{cor:az1}(ii), the proof of Corollary \ref{cor:az1}(iii) follows from the combination of the above isomorphisms \eqref{eq:iso-Azumaya-1} and \eqref{eq:algebras41} with the equivalence \eqref{eq:equivalence1}.

\section{Grothendieck and Voevodsky's conjectures for orbifolds}\label{sec:Voevodsky}

Assume that $k$ is of characteristic $0$. Given a smooth projective
$k$-scheme $ X$ and a Weil cohomology theory $H^\ast$, let us denote
by $\pi_X^i\colon H^\ast( X) \to H^\ast(X)$ the $i^{\mathrm{th}}$
K\"unneth projector, by $Z^\ast(X)_\bbQ$ the $\bbQ$-vector space of
algebraic cycles on $X$ (up to rational equivalence), and by
$Z^\ast(X)_\bbQ/_{\!\sim \otimes \mathrm{nil}}$,
$Z^\ast(X)_\bbQ/_{\!\sim \mathrm{hom}}$, and $Z^\ast(X)_\bbQ/_{\!\sim
  \mathrm{num}}$, the quotients with respect to the smash-nilpotence,
homological, and numerical equivalence relations; consult
\cite[\S3.2]{Andre} for details. Following Grothendieck \cite{Grothendieck} (see
also Kleiman \cite{Kleim1, Kleim}), the standard
conjecture\footnote{The standard conjecture of type $C^+$ is also
  usually known as the {\em sign conjecture}. Note that if $\pi_X^+$ is algebraic, then the odd K\"unneth projector $\pi_X^- := \sum_i \pi_X^{2i+1}$ is also algebraic.} of type $C^+$, denoted
by $C^+(X)$, asserts that the even K\"unneth projector
$\pi^+_X:=\sum_i \pi^{2i}_X$ is algebraic, and the standard conjecture
of type $D$, denoted by $D(X)$, asserts that $Z^\ast(X)_\bbQ/_{\!\sim
  \mathrm{hom}}=Z^\ast(X)_\bbQ/_{\!\sim \mathrm{num}}$. Following
Voevodsky \cite{Voevodsky-IMRN}, the smash-nilpotence conjecture,
denoted by $V(X)$, asserts that $Z^\ast(X)_\bbQ/_{\!\sim
  \otimes\mathrm{nil}}=Z^\ast(X)_\bbQ/_{\!\sim \mathrm{num}}$. 
\begin{remark}[Status]
\begin{itemize}
\item[(i)] Thanks to the work of Grothendieck and Kleiman (see \cite{Grothendieck,Kleim1,Kleim}), the conjecture $C^+(X)$ holds when $X$ is of dimension $\leq 2$, and also for abelian varieties. Moreover, this conjecture is stable under products.
\item[(ii)] Thanks to the work of Lieberman \cite{Lieberman}, the conjecture $D(X)$ holds when $X$ is of dimension $\leq 4$, and also for abelian varieties.
\item[(iii)] Thanks to the work Voevodsky \cite{Voevodsky-IMRN} and Voisin \cite{Voisin}, the conjecture $V(X)$ holds when $X$ is of dimension $\leq 2$. Moreover, thanks to the work of Kahn-Sebastian \cite{KS}, the conjecture $V(X)$ also holds for abelian $3$-folds.
\end{itemize}
\end{remark} 
In \cite{Voev-conj,CD}, the aforementioned conjectures of Grothendieck and Voevodsky were proved in some new cases (\eg\ quadric fibrations, intersections of quadrics, intersections of bilinear divisors, linear sections of Grassmannians, linear sections of determinantal varieties, Moishezon varieties, etc) and extended from smooth
projective $k$-schemes $X$ to smooth proper algebraic stacks $\cX$. Using Theorem \ref{thm:main}, we are now able to  verify these conjectures in the case of ``low-dimensional'' orbifolds:
\begin{theorem}\label{thm:conjectures} Let
  $X$ be a smooth projective $k$-scheme equipped with a
  $G$-action.  \begin{itemize} \item[(i)] If $\mathrm{dim}(X)\leq 2$
    or $X$ is an abelian variety and $G$ acts by group homomorphisms,
    then the conjecture $C^+([X/G])$ holds.  \item[(ii)] If
    $\mathrm{dim}(X)\leq 4$, then the conjecture $D([X/G])$
    holds.  \item[(iii)] If $\mathrm{dim}(X)\leq 2$, then the
    conjecture $V([X/G])$ holds.  \end{itemize} \end{theorem}
\begin{proof}
Recall from \cite[\S4.1]{book} that the category of {\em noncommutative Chow
  motives $\NChow(k)_\bbQ$} is defined as the smallest full, additive,
idempotent complete, symmetric monoidal subcategory of $\NMot(k)_\bbQ$
containing the objects $\U(\cA)_\bbQ$, with $\cA$ a smooth proper dg
category. Let $N\!M \in \NChow(k)_\bbQ$ be a noncommutative Chow
motive. As explained in \cite{Voev-conj,CD} (see also \cite[\S3]{survey}), the conjectures $C^+$,
$D$, and $V$, admit noncommutative analogues
$C^+_{\mathrm{nc}}(N\!M)$, $D_{\mathrm{nc}}(N\!M)$, and
$V_{\mathrm{nc}}(N\!M)$, respectively. Moreover, given a smooth projective $k$-scheme $X$, we have the
equivalences:
\begin{eqnarray*}
C^+(X) \Leftrightarrow C^+_{\mathrm{nc}}(\U(X)_\bbQ) & D(X) \Leftrightarrow D_{\mathrm{nc}}(\U(X)_\bbQ) & V(X) \Leftrightarrow V_{\mathrm{nc}}(\U(X)_\bbQ)\,.
\end{eqnarray*} 
This motivated the extension of Grothendieck and Voevodsky's conjectures from smooth projective schemes $X$ to smooth proper algebraic stacks $\cX$ by setting 
\begin{equation*}
C^+(\cX):=C^+_{\mathrm{nc}}(\U(\cX)_\bbQ) \qquad D(\cX):=D_{\mathrm{nc}}(\U(\cX)_\bbQ) \qquad V(\cX):=V_{\mathrm{nc}}(\U(\cX)_\bbQ)\,.
\end{equation*}
Now, note that in the case where $\cX=[X/G]$, the formula \eqref{eq:formula-main-sub} implies that the noncommutative motive $\U([X/G])_\bbQ$ is a direct summand of $\bigoplus_{\sigma \in \varphi} \U(X^\sigma \times \mathrm{Spec}(k[\sigma]))$. Therefore, since by construction the conjectures $C^+_{\mathrm{nc}}$, $D_{\mathrm{nc}}$, and $V_{\mathrm{nc}}$, are stable under direct sums and direct summands, the above considerations show that in order to prove Theorem \ref{thm:conjectures} it suffices to prove the following conjectures:
\begin{eqnarray*}
C^+(X^\sigma \times \mathrm{Spec}(k[\sigma])) &  D(X^\sigma \times \mathrm{Spec}(k[\sigma]))&  V(X^\sigma \times  \mathrm{Spec}(k[\sigma]))\,.
\end{eqnarray*}
%
\begin{itemize}
\item[(i)] Using the fact that the conjecture $C^+$ is moreover stable under products, it is enough to prove the conjectures $C^+(X^\sigma)$ and $C^+(\mathrm{Spec}(k[\sigma]))$. On the one hand, since $\mathrm{Spec}(k[\sigma])$ is $0$-dimensional, the conjecture $C^+(\mathrm{Spec}(k[\sigma]))$ holds. On the other hand, since the assumptions imply that $\mathrm{dim}(X^\sigma)\leq 2$ or that $X^\sigma$ is an abelian variety, the conjecture $C^+(X^\sigma)$ also holds. 



\item[(ii)] The assumptions imply that $\mathrm{dim}(X^\sigma \times \Spec k[\sigma]) \leq 4$.
\item[(iii)] The assumptions imply that $\mathrm{dim}(X^\sigma \times \Spec k[\sigma]) \leq 2$.
\end{itemize}
The above considerations (i)-(iii) conclude the proof.
\end{proof}

\appendix 

\section{Nilpotency in the Grothendieck ring of an orbifold}\label{sec:appendix}
In what follows, we don't assume that $1/n \in k$.
Given a connected $k$-scheme $X$, it is natural to ask if the elements in the kernel of the rank map $K_0(X) \to \bbZ$ are nilpotent. This is well-known in the case where $X$ is Noetherian and admits an ample line bundle\footnote{Under these strong assumptions, the kernel of the rank map is itself nilpotent.}; consult \cite[Chapter V \S3]{Fulton} for example. The more general case where $X$ is quasi-compact and quasi-separated was proved in \cite[Thm.~2.3]{Azumaya}. In this appendix, we further extend the latter result to the case of (global) orbifolds. 

%
\begin{theorem} \label{thm:dm}
Let $X$ be quasi-compact quasi-separated scheme equipped with a $G$-action. Assume that $1/n\in \cO_X$ and that $\cX:=[X/G]$ is equipped with a finite morphism towards a quasi-compact quasi-separated algebraic space $Y$. Under these assumptions, the elements in the kernel of the induced map
\[
K_0(\cX)\too \prod_{\bar{y}\r Y} K_0(\cX_{\bar{y}})\,,
\]
where the product runs over the geometric points  of $Y$, are nilpotent.
\end{theorem}
\begin{proof} 
The inductive argument used in the proof of \cite[Thm.~2.3]{Azumaya} for quasi-compact quasi-separated schemes holds similarly for quasi-compact quasi-separated algebraic spaces; simply replace the Zariski topology by the Nisnevich topology\footnote{It is well-known that (nonconnective) algebraic $K$-theory satisfies not only Zariski-descent but also Nisnevich-descent.} and \cite[Prop.~3.3.1]{BO} by \cite[Tag 08GL, Lem.~62.8.3]{Stacks}. Therefore, by applying this inductive argument to the algebraic space $Y$, it suffices to prove the following local
statement: {\em let $(R,\mathfrak{m})$ be a local ring with residue field $\kappa:=R/\mathfrak{m}$, $Y:=\mathrm{Spec}(R)$, $y:=\mathrm{Spec}(\kappa)$ the closed point of $Y$, and $A:=\Oscr_{X,y} \# G$ the skew group algebra. Under these notations, for every field extension $l/\kappa$, the induced homomorphism $K_0(A) \to K_0(A_l)$ is injective}. This local statement is a particular case of Lemma \ref{lem:aux} below. 
%
\end{proof}
\begin{lemma}\label{lem:aux} Let $(R,\mathfrak{m})$ be a local ring with residue field $\kappa:=R/\mathfrak{m}$, $A$ an $R$-algebra (which we assume finitely generated as an $R$-module), and $l/\kappa$ a field extension. Under these assumptions, the induced homomorphism $K_0(A) \to K_0(A_l)$ is injective. 
\end{lemma}
\begin{proof}
Note first that every simple $A$-module $V$ is finitely generated as an $R$-module.
If $r\in \mathfrak{m}$, then $rV$ is an $A$-submodule of $V$. Therefore, by Nakayama's lemma applied to $R$, we have $rV=0$. This implies that $\mathfrak{m}A\subset \rad(A)$, where $\rad(A)$ stands for the Jacobson's radical of $A$. Using  Nakayama's lemma once again, we conclude that the induced homomorphism $K_0(A)\to K_0(A_\kappa)$ is injective. Thanks to the above considerations, it suffices then to prove the particular case where $R=\kappa$ is a field and $A$ a finite dimensional $\kappa$-algebra. By replacing $A$ with $A/\!\rad(A)$, we can moreover assume that $A$ is semi-simple, and, using Morita theory, we can furthermore assume that $A$ is a division $\kappa$-algebra. In this latter case, we have an isomorphism $K_0(A)\simeq \bbZ$ induced by the rank map. The proof follows now from the fact that the rank map is preserved under base-change $-\otimes_\kappa l$. 
%
%
\end{proof}
\begin{corollary} \label{cor:nilp}
Let $X$ be a quasi-compact separated $k$-scheme equipped with a $G$-action. Assume that $1/n\in \cO_X$. For every finite field extension $l/k$, the elements in the kernel of the induced homomorphism $K_0([X/G])\r K_0([X_l/G])$ are nilpotent.
\end{corollary}
\begin{proof} Let $\pi:[X/G]\r Y$ be the coarse moduli space of
  $[X/G]$; see \cite{KM1}. Note that the morphism $\pi$ is finite and that $[X_l/G]=[X/G]\times_{Y} Y_l$. Therefore, the proof follows from the combination of Theorem \ref{thm:dm} with the fact that every geometric point of $Y$ factors through a geometric point of $Y_l$.
\end{proof}
\begin{remark} The homomorphism $K_0([X/G])\r K_0([X_l/G])$ is not necessarily injective. However, if $r$ stands for the degree of the (finite) field extension $l/k$, then the induced homomorphism $K_0([X/G])_{1/r}\r K_0([X_l/G])_{1/r}$ is injective.
\end{remark}

\end{document}

\end{proof}